\documentclass[aap,preprint]{imsart}
\RequirePackage[OT1]{fontenc}
\usepackage{graphicx}
\usepackage[T1]{fontenc}
\usepackage{enumerate}
\usepackage{latexsym,amssymb,amsthm,amsxtra}
\usepackage{amsmath,mathrsfs}
\usepackage{stmaryrd}
\usepackage{pstricks}
\usepackage{pst-node}
\usepackage{multido}
\usepackage{tikz}
\usepackage{pgf,pgfarrows,pgfnodes,pgfautomata,pgfheaps,pgfshade,mathrsfs}
\usepackage{hyperref,breakurl}

\startlocaldefs
{\bf}{\it}
\newtheorem{theorem}{Theorem}
\newtheorem{definition}{Definition}
\newtheorem{lemma}{Lemma}
\newtheorem{remark}{Remark}
{\it}
\newtheorem{proposition}{Proposition}
\newtheorem{corollary}{Corollary}

\newtheorem{ex}{Example}{}
\newtheorem{assumption}{Assumption}
\newtheorem{conjecture}{Conjecture}

\def\/{\, | \,}

\def\ind{{\mathchoice {\rm 1\mskip-4mu l} {\rm 1\mskip-4mu l}
{\rm 1\mskip-4.5mu l} {\rm 1\mskip-5mu l}}}

\def\Z{{\mathbb Z}}
\def\boe{{\mathbf e}}
\newcommand{\D}{{\mathbb D}}

\def\ssc{\scriptsize{\text{\textsc{c}}}}

\def\sd{\scriptsize{\text{\textsc{d}}}}

\def\sM{\text{\textsc{m}}}
\def\sL{\text{\textsc{l}}}

\def\sU{\text{\textsc{u}}}

\def\v{{\--}}
\def\pv{{\not\!\!\--}}

\newcommand{\R}{{\mathbb R}}

\newcommand{\maV}{{\mathcal V}}

\def\Z{{\mathbb Z}}
\def\I{{\mathbb I}}

\newcommand\proce[1]{\left(#1,\, t\ge 0\right)}

\newcommand\gre{\mathbf{e}}

\newcommand\maF{{\mathcal F}}

\newcommand\maL{{\mathcal L}}

\newcommand\maI{{\mathcal I}}
\newcommand\maS{{\mathcal S}}

\newcommand\maE{{\mathcal E}}

\newcommand\maM{{\mathcal M}}
\newcommand\maU{{\mathcal U}}
\newcommand\maP{{\mathcal P}}

\newcommand\maN{{\mathcal N}}
\newcommand\maO{{\mathcal O}}
\newcommand\maT{{\mathcal T}}
\newcommand\mGS{{\mathbb G^{\maS}}}

\newcommand{\af}{\alpha}
\newcommand{\lm}{\lambda}
\newcommand{\ep}{\epsilon}

\newcommand{\Ra}{\Rightarrow}
\newcommand{\ra}{\rightarrow}
\newcommand{\deq}{\stackrel{\rm d}{=}}

\newcommand{\qandq}{\quad\mbox{and}\quad}
\newcommand{\qifq}{\quad\mbox{if}\quad}

\newcommand{\qasq}{\quad\mbox{as ~}}

\newcommand{\qforallq}{\quad\mbox{for all }}

\newcommand{\RR}{{\mathbb R}}
\newcommand{\ZZ}{{\mathbb Z}}

\def\tinf{\rightarrow\infty}

\def\barq{\bar{Q}}

\def\on{\stackrel}
\def\wt{\widetilde}

\newcommand{\bes}{\begin{equation*}}
\newcommand{\ees}{\end{equation*}}
\newcommand{\bequ}{\begin{equation}}
\newcommand{\eeq}{\end{equation}}
\newcommand{\bi}{\begin{itemize}}
\newcommand{\ei}{\end{itemize}}
\newcommand{\bsplit}{\begin{split}}
\newcommand{\esplit}{\end{split}}
\newcommand{\bea}{\begin{eqnarray}}
\newcommand{\eea}{\end{eqnarray}}
\newcommand{\beas}{\begin{eqnarray*}}
\newcommand{\eeas}{\end{eqnarray*}}

\newsavebox{\fmbox}

{\end{list}}%

{\end{list}}%

\endlocaldefs
\begin{document}
\begin{frontmatter}

\title{On the Instability of Matching Queues}
\runtitle{Instability of Matching Queues}

\begin{aug}
\author{\fnms{Pascal} \snm{Moyal}\thanksref{m1}\ead[label=e1]{pascal.moyal@utc.fr}}
\and
\author{\fnms{Ohad} \snm{Perry}\thanksref{m2}\ead[label=e2]{ohad.perry@northwestern.edu}}

\runauthor{Moyal and Perry}

\affiliation{Universit\'e de Technologie de Compi\`egne\thanksmark{m1} and
Northwestern University\thanksmark{m2}}

\address{Laboratoire de Math\'ematiques appliqu\'ees, Universit\'e de Technologie de Compi\`egne\\
Industrial Engineering and Management Sciences, Northwestern University \\
\printead{e1}\\
\phantom{E-mail:\ }\printead*{e2}}

\end{aug}

\begin{abstract}
A matching queue is described via a graph, an arrival process, and a
matching policy. Specifically, to each node in the graph there is a corresponding
arrival process of items, which can either be queued or matched
with queued items in neighboring nodes. The matching policy
specifies how items are matched whenever more than one matching is
possible. Given the matching graph and the matching policy, the
stability region of the system is the set of intensities of the
arrival processes rendering the underlying Markov process positive
recurrent. In a recent paper, a condition on the arrival
intensities, which was named {\sc Ncond}, was shown to be necessary
for the stability of a matching queue. That condition can be thought
of as an analogue to the usual traffic condition for traditional
queueing networks, and it is thus natural to study whether it is
also sufficient. In this paper, we show that this is not the case in
general. Specifically, we prove that, except for a particular class
of graphs, there always exists a matching policy rendering the
stability region strictly smaller than the set of arrival
intensities satisfying {\sc Ncond}. Our proof combines graph- and
queueing-theoretic techniques: After showing explicitly, via
fluid-limit arguments, that the stability regions of two basic
models is strictly included in {\sc Ncond}, we generalize this
result to any graph inducing either one of those two basic graphs.
\end{abstract}



\end{frontmatter}

\maketitle

\section{Introduction} \label{secIntro}
We consider a continuous-time matching queueing system in which
items of different classes arrive one by one and
depart in pairs. 
Specifically, we assume that any item is either matched with exactly one other item immediately upon arrival,
and both items leave the system, or is stored in a buffer until it is matched.
Since matchings are pairwise, such a matching model can be represented via an undirected graph,
in which each node represents a class of arrivals, and an edge between two nodes represents that items of the two corresponding classes can be matched together.
A {\em matching policy} describes the matching rule whenever more than one matching is possible for an incoming item.

The model just described is closely related to the discrete-time {\em stochastic matching model} introduced in \cite{MaiMoyal15}, in which
items enter the system at discrete time points, and their class is drawn upon arrival from
a given probability distribution on the set of classes. We elaborate on the relation between the two models below.

\paragraph{Matching models in the literature}

Matching queueing models arise directly in several applications, such as organ transplantation \cite{Davis_kidney} and public-housing assignments \cite{TW08}.
They were also employed in the literature as relaxations for complex many-server queueing systems \cite{AdWe, CKW09},
stochastic processing networks, and assemble-to-order systems; see \cite{GurWa}.
We refer to these references for comprehensive literature reviews of related models.

A more recent application for matching queues is in modeling {\em sharing-economy} (or{\em collaborative consumption}) platforms,
with the most
relevant examples being car-sharing platforms, such as Uber and Lyft, lodging services, such as Airbnb,
and virtual call centers (namely call centers with home-based agents), as considered in, e.g., \cite{GLM15, rouba}.
Since a platform operating in a sharing-economy market must match supply and demand at every instance, possibly
in a multi-region setting, matching queues can be used to model and optimize such platforms; see \cite{Joh15} for a recent application in the car-sharing setting.


The term {\em matching queues} was introduced in \cite{GurWa}.
In that reference, items can be matched by groups of size two or more, and the goal
is to minimize finite-horizon cumulative holding costs. Moreover, the
controller can keep matchable items in storage for more ``profitable'' future matches.
A myopic, discrete-review control is shown to be asymptotically optimal, as the arrival rate grows large.
Thus, both the model and the objectives of \cite{GurWa} are different than ours here, since we consider
pairwise matchings, and are concerned with stability properties.

The stability of matching models operating under
the {\em first-come first-served} (FCFS) policy was studied for several particular graphs in \cite{CKW09}, assuming that
both the arrivals and the departures occur by pairs. (If items arrive one by one, the model can never be stable, as will become clear below.)
A discrete-time Markov chain representation
was employed in combination with Lyapunov techniques, to study the stability (ergodicity) of certain models having bipartite graphs.
An important (and expected) observation from this latter reference is that characterizing the stability region,
namely, the law of the arrival process under which the system is stable, is
nontrivial even for models with relatively simple graphs. 

The matching model in \cite{CKW09} was further analyzed in \cite{BGMa12}.
An alternative Markov representation was introduced,
leading to a more complete picture of the stability problem. General sufficient and necessary conditions for stability of the underlying Markov chain were given,
together with properties of several matching policies. In particular, the stability of any model applying the 'match the longest' policy (under which
an item that has more than one matching option upon arrival is matched with an item from the longest queue) was proved, assuming the necessary condition
for stability holds.

The necessary condition for stability in \cite{BGMa12} was employed in \cite{AdWe} to prove the existence of unique matching rates
for models satisfying a certain ``complete resource pooling'' condition. 
The models considered in \cite{AdWe}
are again bipartite and operate in the FCFS matching policy.
Interestingly, the stationary distribution of the Markov chain in \cite{AdWe} is shown to have a product form.

In \cite{AdWe14} a continuous-time model is considered for a bipartite matching system operating under the FCFS-ALIS (Assign Longest-Idle Server)
policy. In this paper, Markovian service system (i.e., service times are considered) with skill-based routing, are modeled as matching queues,
and the stationary distribution, when it exists, is shown to have a product form. Fluid limits are employed in the overloaded case to prove the existence
of a local steady state.

The graphs in \cite{AdWe,AdWe14, BGMa12, CKW09} are all bipartite.
Therefore, to make the question of stability nontrivial, items are
assumed to arrive in pairs, as was mentioned above. This assumption
is dropped in \cite{MaiMoyal15}, which introduces the aforementioned
(discrete-time) stochastic matching model with general graph
topology. In \cite{MaiMoyal15}, a thorough study of the structure of
the stability region of the model is proposed, partially relying on
the results in \cite{BGMa12}. A natural necessary condition, named
{\sc Ncond}, for the stability of {\bf any} such stochastic matching
model is introduced, implying, in particular, that no model can be
stable if the matching graph is a tree, or more generally, a
bipartite graph. (This explains the assumption that arrivals occur
in pairs in the papers dealing with bipartite graphs cited above. If
items arrive one by one, the system cannot be stable.) In addition,
a particular class of graphs is exhibited (the non-bipartite {\em
separable} ones - see Definition \ref{def:sep} below), for which
{\sc Ncond} is also a sufficient condition for stability. However,
the study of a particular model on a non-separable graph (see
\cite{MaiMoyal15}, p.14) shows that {\sc Ncond} is not sufficient in
general for non-separable graphs. This raises the question of
whether the sufficiency of {\sc Ncond} is true {\em only} for
separable graphs. The Lyapunov-stability techniques that were
employed for the particular model in \cite{MaiMoyal15}, render the
generalization of the arguments in \cite{MaiMoyal15} to a larger
class of non-separable graphs impractical in the discrete-time
settings.


Since a matching queue
As is easily seen, and will be shown in Theorem \ref{thm:stabCont}
below, the stability region of a discrete-time stochastic model can
be studied by embedding it in an appropriate continuous-time model.
Thus, the continuous-time counterparts of the results in
\cite{MaiMoyal15} hold for our matching queues, and vice versa. The
advantage of the continuous-time setting is that powerful
fluid-limit techniques can be employed, which greatly facilitate the
stability analysis.


\paragraph{Stability of Stochastic Networks via Fluid Limits}
The necessary condition for stability {\sc Ncond}, defined in
\eqref{Ncond1} below, can be thought of as an analogue to the usual
traffic condition of standard queueing networks, requiring that the
long-term rate of arrivals to each service station be less than the
long-run output rate at that station. Therefore, our work relates to
the literature on (in)stability of subcritical stochastic networks,
which we briefly review.

Consider a stochastic queueing network with $d \ge 1$ service stations and $K \ge 1$ classes. 
Let $a^k_i$ denote the (long-run) arrival rate
of class-$k$ jobs into service station $i$,
and let $m^k_i$ denote the mean service time of class-$k$ jobs in this station, $1 \le k \le K$.
Then the system is subcritical if
\bequ \label{subcritical}
\rho_i := \sum_k a^k_i m^k_i < 1 \qforallq 1 \le i \le d.
\eeq
It is well-known that Condition (\ref{subcritical}) is not sufficient to ensure stability of stochastic networks in general.
The first examples of this fact are the deterministic Lu-Kumar network \cite{LuKumar}
(and its stochastic counterpart; \S 3 in \cite{BramsonBook}),
and the Rybko-Stolyar network \cite{RySto92}, both of which consider static priority
service policies. A subcritical, multiclass, two-station network, having a Poisson arrival process and exponential service times
in both service stations for all classes, was shown to be unstable under the FIFO discipline in \cite{Bramson_Instable}.
See \cite{BramsonBook} for an elaborate discussion, including a comprehensive literature review of the subject.

Fluid models are arguably the most effective tool to proving that a queueing network is stable, and can also be employed to prove instability of such networks.
Specifically, following \cite{RySto92}, Dai \cite{DaiFluid} showed that, under mild regularity conditions,
if all the (subsequential) fluid limits of the queues, for all possible initial conditions,
converge to $0$ in finite time w.p.1, then the system is stable, in the sense that the underlying queue process is positive Harris recurrent.
We rely on \cite{DaiFluid} to characterize the stability region of specific matching queues
(see Proposition \ref{propPendant_iff} and Corollary \ref{cor:5cycle}
below), but our main result is concerned with proving a general {\em instability} result.
Regarding the use of fluid limits to prove instability of stochastic networks, we mention that
partial converse results to \cite{DaiFluid} exist, e.g., \cite{daiInstablefluid, GamHas2005, meyn1995}.
Here, we build on the theory in \cite[Ch.\ 9]{RobertBook} to prove our main result by
characterizing fluid limits uniquely for appropriate initial conditions,
and showing that those fluid limits do not decrease to the origin in finite time; see Lemma \ref{lmUnstable} below.


An interesting feature of the fluid limits we obtain is that their dynamics are determined by the stationary distribution of a ``fast'' CTMC.
Specifically, if the fluid queue associated with one of the nodes is positive, then the relevant time scale for this queue is
slower than the time scale for the fluid queues that are null.
In the limit, the effect of the ``fast'' (i.e., null) queues on the evolution of the positive fluid queues is averaged-out
instantaneously, a phenomena known as a stochastic {\em averaging principle} (AP) in the literature. 
See \cite{PW-ODE, PWfluid} and the references therein, as well as \cite{Zhang_Chatting, Zhang_wireless} for recent examples of fast averaging
in queueing networks.

\paragraph{Organization of the Paper}
The rest of the paper is organized as follows. In Section \ref{secModel} we elaborate on our model
and introduce the main notation and terms that will be used.
In Section \ref{secMain} we present our main result, Theorem \ref{thm:main}. 
Section \ref{secFluidStable} develops the fluid limit (Theorem \ref{thFWLLN}) 
and Section \ref{sec:Examples} studies models that are key to the proof of Theorem \ref{thm:main}.
Theorem \ref{thm:main} is proved in Section \ref{secProofMain}, building on
the results of Sections \ref{secFluidStable} and \ref{sec:Examples}, and the FWLLN is proved in Section \ref{secFluidProof}.
In Section \ref{sec:U} we present a related instability result for an
alternative matching policy.
Summary of the main paper and directions for future research are presented in \ref{secSummary}.
In addition to the main paper, in Appendix \ref{sec:RG} we demonstrate that our main results have implications to the construction
of matchings on random graphs.

\section{The Model} \label{secModel}
In this section we describe the matching queueing model in detail, after introducing the notation and key terms that we employ.

\subsection{Basic Terms and Notation}

We adopt the usual $\RR$ and $\ZZ$ notation for the sets of real numbers and
integers, respectively. We let $\RR_+$ and $\RR_{++}$ denote, respectively, the sets of non-negative and strictly positive real numbers.
Similarly, $\ZZ_+$ and $\ZZ_{++}$ denote the sets of non-negative and strictly positive integers.
For any two elements $a,b \in \ZZ_+$, we let $\llbracket a,b \rrbracket := \left\{a,a+1,...,b\right\}.$
For a set $A$, we let $|A|$, 
denote the cardinality of $A$ and for any $k \in \ZZ_{++}$, 
$A^k$ denotes the set of $k$-dimensional vectors with components in $A$.
For any $i \in \llbracket 1,k \rrbracket$, the $i$th vector of the canonical basis of $\RR^k$ is denoted by $\boe_i$, namely, 
$\boe_i$ has $1$ in its $i$th coordinate and $0$ elsewhere.
For any subset $J \subset \llbracket 1,k \rrbracket$ and $x\in \RR^k$,
we use the notation $x_J$ for the restriction of $x$ to its coordinates corresponding to the indices of $J$.

For an interval $I \subset [0, \infty)$, let $\D^d(I)$ denote the space of $\RR^d$-valued functions on $I$ that are right continuous and have limits from the left everywhere,
endowed with the standard Skorohod $J_1$ topology \cite{Billingsley}. .
To simplify notation, we write, e.g., $\D^d(a,b)$ instead of $\D^d((a,b))$,
and $\D(I)$ for $\D^1(I)$ (we remove the superscript when $d=1$).
We omit the interval from the notation whenever it can be taken to be an arbitrary compact interval, e.g., convergence in $\D^d$
holds over compact subintervals of $[0, \infty)$. We write $\mathbb C^d(I)$ for the subspace of continuous functions on $I$,
with $\mathbb C := \mathbb C^1$.

\paragraph{Random Variables and Processes}
We work on the probability space $\left(\Omega,\maF,P\right)$.
We write $\deq$ to denote equality in distribution and $\Rightarrow$ to denote convergence in distribution.
For a sequence of real-valued random variables $\{Y^n : n \in \ZZ_{++}\}$ we write $Y^n \Ra \infty$ if $P(Y^n > M) \ra 1$ as $n\tinf$, for any $M > 0$.
The fluid-scaled version of a sequence of stochastic processes $\{Y_n : n \in \ZZ_{++}\}$ is denoted by $\bar{Y}^n := Y^n / n$.

For two real-valued stochastic processes $X$ and $Y$ we write $X
\le_{st} Y$ if $X$ is smaller than $Y$ in {\em sample-path
stochastic order}, namely, if it is possible to construct two
processes $\tilde X$ and $\tilde Y$ on a common probability space,
such that $\tilde X \deq X$, $\tilde Y \deq Y$, and the sample paths
of $\tilde X$ lie below those of $\tilde Y$ with probability 1
(w.p.1 for short). When $X$ and $Y$ are $\RR^d$-valued, $d > 1$, $X
\le_{st} Y$ means that $X_i \le_{st} Y_i$, for all $1 \le i \le d$.

\paragraph{Graph-Related Terminology}

In addition to the notation, we introduce basic terms of graph theory that will be used below.
A graph $G$ is denoted by $G = (\maV, \maE)$, where $\maV$ and $\maE$ are the set of nodes and edges, respectively.
The nodes of $G$ are labeled arbitrarily in $\llbracket 1,|\maV| \rrbracket$,
and we often identify $\maV$ with $\llbracket 1,|\maV| \rrbracket$.
We write $i \v j$ whenever $(i,j)\in \maE$, namely, nodes $i$ and $j$ are connected by an edge, and $i \pv j$ otherwise.
If $i \v j$, then these two nodes are said to be {\em neighbors}.
All the graphs considered in this paper are {\em non-oriented}, i.e. $i\v j$ if and only if $j \v i$
(and the edges $(i,j)$ and $(j,i)$ are indistinguishable) and {\em simple}, {\em i.e.} there is no edge connecting a node to itself
($i \pv i$ for any $i \in \maV$), and two nodes are connected by at most one edge.

For any subset $A$ of $\maV$, we let $\maE(A)$ be the set of neighbors of all the nodes in $A$, i.e.,
\bes
\maE(A) := \{j \in \maV : i \v j, \text{ for some } i \in A\}.
\ees
We write $\maE(i)$ for the neighbors of a single node $i$ (instead of $\maE(\{i\})$).

The graph $\breve G=(\breve \maV,\breve \maE)$ is said to be a {\em
subgraph} of $G=(\maV,\maE)$ if $\breve \maV \subset \maV$ and
$\breve \maE \subset \maE$. We say that $G$ {\em induces} the
subgraph $\breve G$, whenever $\breve \maE$ equals the restriction
of $\maE$ to $\breve\maV^2$ (recall that $A^2=A\times A$ for any set
$A$). In other words, if $(i,j)\in \maE$, then $(i,j)\in \breve
\maE$, for all $i,j \in \breve \maV$. In that case, $\breve G$ is
said to be {\em induced by $\breve {\maV}$ in $G$}.

\noindent For any $p \in \ZZ_{++}$, a {\em cycle of length $p$} (or {\em $p$-cycle}, for short) is a graph
$G=(\maV,\maE)$ such that $|\maV|=p$ and any node in $\maV$ has exactly two neighbors. In other words, we can
label the nodes of $G$ as $i_1,i_2,...,i_p$, such that
\[i_1 \v i_2,\,i_2 \v i_3,\, ... \,,i_{p-1}\v i_p\mbox{ and }i_p \v i_1.\]
We say that the $p$-cycle is {\em odd} if its length $p$ is an odd number.

\noindent The {\em complement} graph of $G$ is the graph $\bar
G=(\bar \maV,\bar \maE)$ such that $\bar \maV=\maV$ and $\bar\maE =
\maV^2 \setminus (\mathcal D \cup \maE),$ where $\mathcal D$ is the
diagonal of $\maV^2$, namely, $\mathcal D := \{(i,i) : i \in
\maV\}$. For $q\ge 2$, the graph $G=(\maV,\maE)$ is said to be
$q$-partite, $q \in \ZZ_{++}$, if there exists a partition $\{\maV_i
: 1 \le i \le q\}$ of $\maV$ such that
$$\maE \subset \bigcup_{i,j\in \llbracket 1,q \rrbracket\,: i\ne j} \maV_i \times \maV_j.$$
In other words, in a $q$-partite graph, every edge links two nodes in two distinct subsets of the partition.
A $2$-partite graph is called {\em bipartite}.

The {\em complete graph} $G=(\maV,\maE)$ is such that
$\maE=\maV^2\setminus \mathcal D$. A {\em clique} of a graph $G$ is a complete subgraph of $G$. The graph $G=(\maV,\maE)$
is said to be {\em connected} if for any $i,j \in \maV$, there exists a path from $i$ to $j$, i.e.
a subset $\{i=i_1,i_2,...,i_q=j\} \subset \maV$ such that $i_{\ell} - i_{\ell+1}$ for any
$\ell \in \llbracket 1,q-1 \rrbracket.$

An {\em independent set} of a graph $G$ is a non-empty subset $\maI \subset \maV$ such that $i \pv j$, for all $i,j \in \maI$.
We let $\I(G)$ denote the set of all independent sets of $G$. Notice that when $G$ is simple, any node is an independent set,
so that $\I(G)$ is nonempty. We say that the independent set $\maI$ of $G$ is {\em maximal} if, 
for any $j\in \maV\setminus \maI$, we have $i_j \v j$ for some $i_j$ in $\maI$. In other words, $\maI \cup \{j\}$ is not an independent set,
for any $j \in \maV \setminus \maI$.

\subsection{Matching Queues} \label{subsecModel}

The {\em matching queue} associated with a graph $G = (\maV, \maE)$,
an arrival-rate vector $\lambda := (\lm_1, \dots, \lm_{|\maV|})$ and a matching policy
$\Phi$, is defined as follows. Each node of the simple graph $G$ (which we call {\em matching graph}) is associated with a class of items,
and items of each class $i \in \maV$ arrive to the system in accordance with a Poisson process $N_i$ having intensity $\lm_i > 0$.
We also write
\begin{equation}
\label{eq:deflambda}
\bar\lm := \sum_{i \in \maV} \lm_i  \qandq \bar\lambda_A := \sum_{i\in A} \lambda_i,\quad A \subset \maV.
\end{equation}
We assume that all $|\maV|$ Poisson arrival processes are independent.
Class-$i$ items can be matched with class-$j$ items if and only if $i \v j$, i.e., there is an edge between the two nodes
$i,j \in \maV$. We emphasize that the matching graphs $G$ we consider are simple, so that items from the same class cannot be matched together.

Upon arrival, a class-$i$ item is either matched with exactly one item from a class $j$ such that $i \-- j$, if any such item is available,
or is placed in an infinite buffer. Matched items leave the system immediately.
We refer to the buffer content associated with each class $i$
as the class-$i$ queue, and denote the associated class-$i$ queue process by $Q_i := \{Q_i(t) : t \ge 0\}$. Specifically,
for all $t \ge 0$, $Q_i(t)$ is the number of the class-$i$ items in queue at time $t$.
Let 
\begin{equation}
\label{eq:defQ}
Q=\left(Q_1,...,Q_{|\maV|}\right)
\end{equation}
denote the $|\maV|$-dimensional queue process of the system.
For $t\ge 0$ and $A \subset \maV$, we let $Q_A(t)$ be the
restriction of $Q(t)$ to its coordinates in $A$.

Upon arrival to the system, a class-$i$ item may find several
possible matches, whenever more than one neighboring class has items
queued. A {\em matching policy} is the rule specifying how to
execute matchings in such cases. We say that a matching policy
$\Phi$ is {\em admissible} if matchings always occur when possible,
and decisions are made solely on the value of the queue
process $Q$ at arrival epochs. (We note that a larger class of policies was considered in \cite{MaiMoyal15}.)
Consequently, under an admissible matching policy, the queue process
$Q$ is a CTMC and $Q_i(t)Q_j(t) = 0$ for all $i, j \in \maV$ such
that $i \v j$ and all $t\ge 0$.


An admissible matching policy is of
{\em priority} type if for any node $i$, the set $\maE(i)$ is a-priori ordered: $\maE(i)=\left(i_1, i_2, \dots, i_{|\maE(i)|}\right)$,
so that, at any time $t$ in which a class-$i$ item enters the system, matching occurs with a class-$i_m$ item,
where $m=\min\left\{\ell \in \llbracket 1,|\maE(i)| \rrbracket:\,Q_{i_\ell}(t) > 0\right\}.$
An important example of a non-priority admissible matching policy is
{\em Match the Longest}, denoted by \textsc{ml}, which was
introduced in \cite{BGMa12} for the bipartite matching queue, and in
\cite{MaiMoyal15} for the general matching queue in discrete time.
According to \textsc{ml}, an arriving item of class $i$ is matched
with an item of the class in $\maE(i)$ that has the longest queue at
that time, where ties are broken according to a uniform draw.
%

Clearly, the initial queue length $Q(0)$, together with $G$,
$\lambda$ and the matching policy $\Phi$ fully determine the
distribution of $Q$. We thus characterize the system by the triple
$(G,\lm,\Phi)_{\ssc}$ (where we append the subscript $\ssc$ to
denote a {\em continuous-time} model, as opposed to the one in
discrete time, which will be denoted with a subscript $\sd$).

\paragraph{Stability of a matching queue}
The matching queue $(G,\lm,\Phi)_{\ssc}$ is said to be {\em
stable} if the corresponding CTMC $Q$ is positive recurrent, and unstable
otherwise.
\begin{definition} The {\em stability region} corresponding to
the connected graph $G$ and the matching policy $\Phi$ is the set
\[\Bigl\{\lambda\in \RR_{++}^{|\maV|}:\,(G,\lm,\Phi)_{\ssc}\mbox{ is
stable}\Bigl\}.\]
\end{definition}
We also say that node $i \in \maV$ is unstable if, for some
initial condition, the mean time for its associated queue to empty
is infinite. Otherwise, the node is stable.

\subsection{The Necessary Condition for Stability {\sc Ncond}}
It is natural to ask what is the analogue of \eqref{subcritical} in the context of matching queues.
Clearly, it must hold that, for each node $i$,
\bequ
\label{notNcond}
\lm_i < \bar\lm_{\maE(i)}.
\eeq
(Recall the notational convention in (\ref{eq:deflambda}).) However, it is easy
to see that \eqref{notNcond} can hold for matching queues having
bipartite matching graphs, for example, although such models are
never stable. (See the discussion following the proof of Theorem
\ref{thNcond} below.) Thus, a necessary condition for stability
should be stronger than \eqref{notNcond}.

To gain intuition, we contrast two simple examples for matching
graphs, the triangle and the simplest graph including a triangle,
which we shall call the ``pendant graph'', depicted in Figures
\ref{Fig:triangle} and \ref{Fig:pendant}, respectively. For the
triangle, it is straightforward that the corresponding matching
queue is stable under \eqref{notNcond} for any admissible matching
policy. Indeed, at most one of the three queues is positive at any
given time, and the drift of any positive queue is necessarily
negative under \eqref{notNcond}. Now consider the pendant graph with
the priority matching policy in Figure \ref{Fig:pendant}, under
which node $3$ gives strict priority to nodes $1$ and $2$, i.e. a
class-$3$ arrival who finds items in node $4$ and in one of the
remaining two nodes, say node $1$, will be matched with the
class-$1$ item. This priority policy is depicted by the arrows in
Figure \ref{Fig:pendant}. Under this policy, node $4$ may be
unstable despite the fact that $\lm_4 < \lm_3$. Indeed, for node $4$
to be stable, we must have an adequate number of class-$3$ arrivals
so that, even though $\lm_3 < \lm_1 + \lm_2 + \lm_4$, sufficiently
many class-$3$ items are left to be matched with all the class-$4$
items in the long run. Since many class $1$ and class $2$ items will
be matched with each other, it is intuitively clear that, in
addition to \eqref{notNcond}, we must require that
$$\lm_4 + \lm_1 < \lm_3 + \lm_2 \qandq \lm_4 + \lm_2 < \lm_3 + \lm_1.$$

\begin{figure}[h!]
  \hfill
  \begin{minipage}[t]{.45\textwidth}
    \begin{center}
\begin{tikzpicture}
 \draw[-] (2,2) -- (1,1);
 \draw[-] (1.6,1.6) -- (1.4,1.4);
 \draw[-] (2,2) -- (3,1);
 \draw[-] (2.4,1.6) -- (2.6,1.4);
 \draw[-] (1,1) -- (3,1);
 \fill (2,2) circle (2pt) node[right] {\small{3}} ;
 \fill (1,1) circle (2pt) node[below] {\small{1}} ;
 \fill (3,1) circle (2pt) node[below] {\small{2}} ;
 \end{tikzpicture}
 \caption[smallcaption]{Triangle}
 \label{Fig:triangle}
    \end{center}
  \end{minipage}
  \hfill
  \begin{minipage}[t]{.45\textwidth}
    \begin{center}
      \begin{tikzpicture}
\draw[-] (2,3) -- (2,2);
\draw[-] (2,2) -- (1,1);
\draw[->, thick] (1.6,1.6) -- (1.4,1.4);
\draw[-] (2,2) -- (3,1);
\draw[->, thick] (2.4,1.6) -- (2.6,1.4);
\draw[-] (1,1) -- (3,1);
\fill (2,3) circle (2pt) node[right] {\small{4}} ;
\fill (2,2) circle (2pt) node[right] {\small{3}} ;
\fill (1,1) circle (2pt) node[below] {\small{1}} ;
\fill (3,1) circle (2pt) node[below] {\small{2}} ;
\end{tikzpicture}
\caption[smallcaption]{Pendant graph}
\label{Fig:pendant}
    \end{center}
  \end{minipage}
  \hfill
\end{figure}

This suggests that one needs to consider the arrival rates to
subsets of non-neighboring nodes, and require that those rates are smaller than the arrival rates to the neighborhoods of those subsets.
Therefore, for any matching graph $G$, we define
\bes 
 \textsc{Ncond}_{\ssc}(G) := \Bigl\{\lambda\in \RR_{++}^{|\maV|}: \bar\lambda_\maI < \bar\lambda_{\maE(\maI)} \text{ for all } \maI \in \I(G)\Bigl\},
\ees where we recall that $\I(G)$ is the set of independent sets of
$G$. We say that {\sc Ncond} holds for $G$ and $\lambda$ if \bequ
\label{Ncond1} \lm \in\textsc{Ncond}_{\ssc}(G). \eeq

\begin{theorem}{$($necessary condition for stability of matching queues$)$} \label{thNcond}
Let $G$ be a connected graph and $\Phi$ an admissible matching
policy. Then the stability region corresponding to $G$ and $\Phi$
is included in $\textsc{Ncond}_{\ssc}(G)$.
\end{theorem}

Before proving Theorem \ref{thNcond}, let us briefly describe the (discrete-time) stochastic matching model introduced
in \cite{MaiMoyal15}. Given a graph $G$ and an admissible matching policy $\Phi$,
the discrete-time stochastic matching model $\left(G,\mu,\Phi\right)_{\sd}$ is defined similarly to the matching queue $\left(G,\lm,\Phi\right)_{\ssc}$,
except that items enter the system one by one, at any discrete time $n \in \ZZ_{++}$.
Assuming that the sequence of classes of the items entering the system
is independent and identically distributed with a common probability measure $\mu$ on $\maV$,
$\left(G,\mu,\Phi\right)_{\sd}$ is represented by the $\ZZ_+^{|\maV|}$-valued Discrete-Time Markov Chain (DTMC)
$U := \{U(n) : n \ge 1\}$, where for any $i \in \maV$ and any $n \in \ZZ_{++}$,
$U_i(n)$ counts the number of items of class $i$ in the buffer at time $n$.
Then, letting $\{N(t) : t \ge 0\}$ denote the superposition of  
of the Poisson arrival processes $N_1,...,N_{|\maV|}$, we have $Q(t) = U(N(t))$, $t \ge 0$, due to uniformization, implying that $Q$ is positive recurrent if and only if $U$ is. 
\begin{proof}[Proof of Theorem \ref{thNcond}]
Define the following set of probability measures $\mu$ on $\maV$,
\begin{equation*} 
\textsc{Ncond}_{\sd}(G) := \Bigl\{\mu\mbox{ with support }\maV : \mu(\maI) < \mu(\maE(\maI)) \text{ for all } \maI \in \I(G)\Bigl\},
\end{equation*}
and for any $\lambda$, define the probability measure
\begin{equation}
\label{eq:defmucont}
\mu_\lambda(i) := \lambda_i/\bar\lambda, \quad i\in\maV.
\end{equation}
Then for any graph $G$ it holds that, if
$\lambda \in \textsc{Ncond}_{\ssc}(G)$, then $\mu_\lambda \in \textsc{Ncond}_{\sd}(G)$.
The statement of the theorem thus follows from Proposition 2 in \cite{MaiMoyal15}, and the fact that $Q$ is positive recurrent if and only if
the DTMC $U$ corresponding to $\left(G,\mu,\Phi\right)_{\sd}$ is.
\end{proof}

An immediate consequence of Theorem \ref{thNcond} is that matching queues $(G,\lm,\Phi)_{\ssc}$
having a connected bipartite graph $G$ are never stable. Indeed, if $\maI_1\cup \maI_2$ denotes the bipartition of $\maV$ into maximal independent sets, then
(\ref{Ncond1}) implies that $\bar\lm_{\maI_1} < \bar\lm_{\maI_2}$ and $\bar\lm_{\maI_1} > \bar\lm_{\maI_2}$, so that $\textsc{Ncond}_{\ssc}(G)$ is empty.
In $(i)$ of Theorem \ref{thm:stabCont} below we show that the converse of this result also holds.

In ending we remark that (\ref{Ncond1}) is equivalent to \eqref{notNcond} for the triangle in Figure \ref{Fig:triangle} or, more generally, for any complete
graph. Aside from this case, condition \eqref{Ncond1} is always strictly stronger (and harder to verify) than \eqref{notNcond}.
It is therefore significant that it can be verified in $O(|\maV|^3)$ time; see Proposition 1 in \cite{MaiMoyal15}.

\section{The Main Result} \label{secMain}

\subsection{Separable graphs}
\label{subsec:separable}
 The notion of a {\em separable graph},
introduced in \cite{MaiMoyal15}, will play a crucial role in what
follows.
\begin{definition}\label{def:sep}
A graph $G=\left(\maV,\maE\right)$ is said to be {\em separable of order $q$}, $q\ge 2$, if there exists a
partition of $\maV$ into maximal independent sets $\maI_1,\dots, \maI_q$, such that $u \v v$ for all $u \in \maI_i$ and $v \in \maI_j$, for all $i \ne j$.
\end{definition}

Equivalently, $G$ is separable of order $q$ if its complement graph
can be partitioned into $q$ disjoint cliques. Notice that a
separable graph of order $2$ is bipartite, whereas a separable graph
or order $3$ or more is non-bipartite.

As we now demonstrate, separable and complete
graphs are closely related. First observe that, for any $q
> 0$, the complete graph of size $q$ is separable of order $q$.
Conversely, any separable graph of order $q$ can be related to the
complete graph of size $q$ in the following way. Let
$G=\left(\maV,\maE\right)$ be a separable graph of order $q$, and
$\maI_1,...,\maI_q$ be its maximal independent sets. Observe
that, as $G$ is separable, the binary relation "$\pv$" is an
equivalence relation in $\maV$. If we `contract' $G$ by ``merging''
all the nodes in each equivalence class (i.e., maximal independent
set), so that each node in the contracted graph represents an
independent set in $G$, and merge all the edges emanating from
merged nodes that point to the same nodes, we obtain the complete
graph $\tilde G$ of size $q$; see Figure \ref{Fig:separable}.

 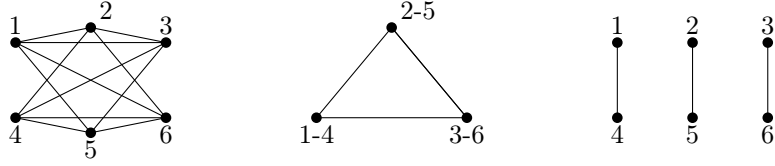
\begin{figure}[h!]
 \begin{center}
 \begin{tikzpicture}
\draw[-] (-2,1) -- (-1,-0.2); \draw[-] (-2,1) -- (0,0);
\draw[-](-2,1) -- (-1,1.2); \draw[-] (-2,1) -- (0,1); \draw[-]
(-1,1.2) -- (-2,0); \draw[-] (-1,1.2) -- (0,0); \draw[-] (-1,1.2) --
(0,1); \draw[-] (0,1) -- (-2,0); \draw[-] (0,1) -- (-1,-0.2);
\draw[-] (-2,0) -- (-1,-0.2); \draw[-] (-2,0) -- (0,0); \draw[-]
(-1,-0.2) -- (0,0); \fill (-2,1) circle (2pt) node[above]
{\small{1}} ; \fill (-1,1.2) circle (2pt) node[above right]
{\small{2}} ; \fill (0,1) circle (2pt) node[above] {\small{3}} ;
\fill (-2,0) circle (2pt) node[below] {\small{4}} ; \fill (-1,-0.2)
circle (2pt) node[below] {\small{5}} ; \fill (0,0) circle (2pt)
node[below] {\small{6}} ;
 \draw[-](3,1.2) -- (2,0); \draw[-] (3,1.2) -- (4,0); \draw[-]
(2,0) -- (4,0); \draw[-] (3,1.2) -- (4,0); \fill (3,1.2) circle
(2pt) node[above right]{\small{2-5}}; \fill (2,0) circle (2pt)
node[below] {\small{1-4}} ; \fill (4,0) circle (2pt) node[below]
{\small{3-6}} ;
\draw[-] (6,1) -- (6,0);
\draw[-] (7,1) -- (7,0);
\draw[-] (8,1) -- (8,0);
\fill (6,1) circle (2pt) node[above] {\small{1}} ;
\fill (7,1) circle (2pt) node[above] {\small{2}} ;
\fill (8,1) circle (2pt) node[above] {\small{3}} ;
\fill (6,0) circle (2pt) node[below] {\small{4}} ;
\fill (7,0) circle (2pt) node[below] {\small{5}} ;
\fill (8,0) circle (2pt) node[below] {\small{6}} ;
 \end{tikzpicture}
 \caption{Separable graph of order $3$ (left); its merged complete graph (middle), and
 its complement graph (right).}
 \label{Fig:separable}
 \end{center}
 \end{figure}

\subsection{Preliminaries} \label{subsecPre}
It follows from Theorem \ref{thNcond} that
(\ref{Ncond1}) is a necessary condition for stability of the system
, regardless of the matching policy $\Phi$.
It is then natural to investigate what are
the matching policies for which 
(\ref{Ncond1}) is also a {\em sufficient} condition for stability of the matching queue.

\begin{definition}
Let $G$ be a connected graph. An admissible matching policy $\Phi$
on $G$ is {\em maximal} if the stability region corresponding
to $G,\Phi$ coincides with $\textsc{Ncond}_{\ssc}(G)$.
\end{definition}

\begin{definition}
A connected graph $G$ is said to be
\begin{itemize}
\item {\em matching-stable} if $\textsc{Ncond}_{\ssc}(G)$ is non-empty and
all admissible matching policies on $G$ are maximal;
\item {\em matching-unstable} if the set $\textsc{Ncond}_{\ssc}(G)$ is empty.
\end{itemize}
\end{definition}
In other words, if $G$ is matching-stable the matching queue
$(G,\lambda,\Phi)_{\ssc}$ is stable for any admissible $\Phi$ and
any $\lambda \in \textsc{Ncond}_{\ssc}(G)$. If $G$ is
matching-unstable, the matching queue $(G,\lambda,\Phi)_{\ssc}$ is
unstable for any admissible $\Phi$ and any arrival-rate vector
$\lambda$. Clearly, a graph $G$ might be neither matching-stable nor
matching-unstable.

We have the following consequence to Theorem 2 in \cite{MaiMoyal15}.
\begin{theorem}
\label{thm:stabCont}
For any connected graph $G$ the following hold.\\
\noindent (i) $G$ is matching-unstable if and only if it is bipartite;\\
\noindent (ii) If $G$ is non-bipartite, then the discipline
{\sc ml} is maximal; \\
\noindent (iii) If $G$ is separable of order $q \geq 3$, then it is matching-stable. 
\end{theorem}

\begin{proof}
The arguments in the proof of Theorem \ref{thNcond} are again
employed to apply the results for the discrete-time model in
\cite{MaiMoyal15} to the continuous-time model considered here. In
particular, by Theorem 1 in \cite{MaiMoyal15}, the set {\sc
Ncond}$_{\ssc}(G)$ is non-empty if and only if $G$ is non-bipartite.
Hence, the statements $(i)$, $(ii)$ and $(iii)$ follow,
respectively, from (16), (17) and (18) in \cite{MaiMoyal15}.
\end{proof}

We make the following observation: For a matching
queue on a complete graph, all admissible matching policies are
equivalent. Indeed, at most one class of items can be present in
queue at any given time, so that an arriving item has no more than one choice for matching.
Thus, the fact that all non-bipartite separable graphs are
matching-stable is not surprising, given their relation to complete graphs, as described above.

Let $G$ be a separable graph of order $q \ge 3$. Consider a matching
queue $(G,\lm,\Phi)_{\ssc}$ on $G$, where $\Phi$ is admissible, and
the matching queue $(\tilde G,\tilde \lambda,\tilde \Phi)_{\ssc}$,
where $\tilde G$ is the ``merged'' complete graph of size $q$
obtained from $G$, as described in Section \ref{subsec:separable},
$\tilde\Phi$ is an arbitrary admissible policy and
\begin{equation*}
 \tilde \lambda_j = \bar \lambda_{\maI_j},\quad j\in \llbracket 1,q\rrbracket.
\end{equation*}
(We add a tilde to all parameters associated with $\tilde G$.) Since
$\tilde G$ is complete, the matching policy $\tilde \Phi$ is
irrelevant, as long as it is admissible. In $(G,\lm,\Phi)_{\ssc}$,
only items of classes belonging to the same maximal independent set
can be present in queue at any given time. Fix a time point $t$ at
which the system is non-empty, and let $\maI_{\ell}$ be the (unique)
maximal independent set having a non-empty queue at $t$. Suppose
that an item enters the system at $t$. We have the following
alternatives:
\begin{itemize}
\item If the new arriving item is of a class belonging to $\maI_{\ell}$,
then no matching occurs and the item joins the queue.
\item If the new arrival is an item of a class $k \in \maI_{m}$, $\ell \ne m$, then
(no matter what the matching policy $\Phi$ is), the entering item
will be matched with an item from a class in $\maI_{\ell}$, and
$\Phi$ only determines the class of its match in $\maI_{\ell}$ (as
several classes in $\maI_{\ell}$ may have a non-empty queues at
$t$).
\end{itemize}

\noindent
Consequently, if the initial conditions satisfy
$\sum_{k\in \maI_j} Q_k(0) \deq \tilde Q_j(0)$, $j \in \llbracket
1,q\rrbracket$,
then
\begin{equation}
\label{eq:sepcompQ} \left\{\sum_{k\in \maI_j} Q_k(t) : t \ge
0\right\} \deq \{\tilde Q_j(t) : t\ge 0\}, \quad j\in \llbracket 1,q
\rrbracket.
\end{equation}

We conclude that, for {\em any} matching policy $\Phi$, if one
adopts a `macroscopic' view of the matching queue
$(G,\lambda,\Phi)_{\ssc}$, by only keeping track of the maximal
independent set present in queue at any time (there is at most one),
and not of the particular classes of the items, then the matching
queue on $G$ amounts to a matching queue on $\tilde G$ having an
arbitrary matching policy $\tilde{\Phi}$. As the latter model is
stable at least for the policy {\sc ml} since $\tilde G$ is
non-bipartite (by assertion (ii) of Theorem \ref{thm:stabCont}
above), it is stable under any policy $\tilde{\Phi}$. From
(\ref{eq:sepcompQ}), the matching queue $(G,\lambda,\Phi)_{\ssc}$ is
stable regardless of the matching policy $\Phi$, which is exactly
assertion (iii) of Theorem \ref{thm:stabCont}.

\subsection{The Main result} \label{subsec:main}
Assertion (i) in Theorem \ref{thm:stabCont} identifies the
class of graphs rendering any matching queue
unstable, regardless of the matching policy. By Assertion (ii) of the theorem, any
matching queue $\left(G,\lambda,\sM\sL\right)_{\ssc}$ on a
non-bipartite graph $G$ is stable, provided $\lambda$ satisfies $\textsc{Ncond}$.
Assertion (iii) presents a class of graphs (the non-bipartite
separable ones) that are matching-stable, namely, for any matching
policy and arrival-rate vector satisfying {\sc Ncond} the system is
stable.
Together, these results raise the question of whether the choice of the matching policy matters in terms of stability
for graphs that are non-separable and non-bipartite, i.e., for graphs for which
at least the discipline {\sc ml} is maximal. 
The simplest such graph, namely, the pendant
graph depicted in Figure \ref{Fig:pendant}, was considered in
Section 5 of \cite{MaiMoyal15},
where it is shown that this graph is not
matching-stable. In particular, it was shown in \cite{MaiMoyal15} that, for a symmetric matching queue (with $\lm_1 = \lm_2$)
there exists a matching policy for which the stability
region is strictly included in {\sc Ncond}$(G)$. 
Our main result, Theorem \ref{thm:main} below, provides a significant generalization of this result for a much larger class of graphs;
en route, we also prove generalized versions of the results in \cite[Section 5]{MaiMoyal15}.

To present our main result,
let $\mathscr{G}_7$ denote the set of all connected graphs inducing an odd cycle of size $7$ or more,
but no $5$-cycle and no pendant graph, and let $\mathscr{G}_7^c$ denote its complement in the set of connected graphs.


\begin{theorem} \label{thm:main}
The only matching-stable graphs in $\mathscr{G}_7^c$ are separable of order $3$ or more.
\end{theorem}

In other words, except for the special case of graphs inducing an
odd cycle of size 7 or more, but {\em no pendant graph and no
5-cycle}, the {\em only} matching-stable graphs are the
non-bipartite separable graphs (i.e., separable graphs of order $3$
or more). Therefore, separability of order at least $3$ is not only
{\em sufficient}, but also {\em necessary}, at least in $\mathscr
G_7^c$, for the stability of any matching queue under
\textsc{Ncond}. We conjecture that, among connected non-bipartite
graphs, separability and matching-stability are equivalent, or in
other words, that no graph in $\mathscr G_7$ is matching-stable.
(The two statements are equivalent since all graphs in $\mathscr
G_7$ are non-separable; see Lemma \ref{Lm:BipSepCycles}(ii) below.)
Even though we were not able to prove this result, we provide key
steps in that direction; see Section \ref{secSummary} below.

Applying again the arguments of the proof of Theorem \ref{thNcond}, we obtain the following immediate corollary to Theorem \ref{thm:main}.

\begin{corollary}
Theorem \ref{thm:main} also holds for the discrete-time matching
model. In particular we have the following partial converse of
assertion (18) in \cite{MaiMoyal15}: if $G \in \mathscr G_7^c$ is
such that any discrete-time matching model $(G,\mu,\Phi)_{\ssc}$ is
stable for $\mu \in \textsc{Ncond}_{\sd}(G)$, then $G$ is separable
of order $q \ge 3$.
\end{corollary}

\subsection{Strategy of the Proof of Theorem \ref{thm:main}} \label{secStrategy}
To prove Theorem \ref{thm:main}, we fix a non-bipartite and {\em
non-separable} graph $G$ in $\mathscr G_7^c$, and show the existence of a non-maximal {\em priority} matching policy $\Phi$.

The proof hinges on the following fact, which will be proved in
Section \ref{secInduced} (statement $(i)$ in Lemma
\ref{Lm:BipSepCycles}): {\em any connected, non-bipartite and
non-separable graph induces a pendant graph or an odd cycle of
length $5$ or more.} Consequently, as $G$ belongs to $\mathscr
G_7^c$, it induces a graph $\breve G$ which is either a pendant
graph or a 5-cycle. The remainder of the proof follows two main
steps:
\begin{enumerate}
\item In Section \ref{sec:Examples} we construct a non-maximal matching policy $\breve \Phi$
on the induced graph $\breve G$ (addressing successively the cases
$\breve G=$ pendant graph and $\breve G=$ 5-cycle), by providing an
arrival-rate vector $\breve \lambda \in \textsc{Ncond}_{\ssc}(\breve
G)$ such that $\left(\breve G, \breve \lambda, \breve
\Phi\right)_{\ssc}$ is unstable. In both cases, the instability of
the system is shown using the fluid-limit arguments developed in
Section \ref{secFluidStable}.
\item We then prove that the instability of the matching queue on the
induced graph $\breve G$ implies instability of the matching queue
on the inducing graph $G$, by showing a ``non-chaoticity" property
in Section \ref{subsec:nonchaotic}. In particular, we show that the
influence of the arrivals to nodes of the complement of $\breve G$
in $G$ can be bounded such that the unstable node in $\left(\breve
G, \breve \lambda, \breve \Phi\right)_{\ssc}$ remains unstable in
$\left(G,\lambda,\Phi\right)_{\ssc}$, for a well-chosen arrival-rate
vector $\lambda \in \textsc{Ncond}_{\ssc}(G)$ and a
well-chosen matching policy $\Phi$.
\end{enumerate}





\section{Fluid Stability} \label{secFluidStable}

We now take a detour to develop the fluid limit which will be used in the proof of Theorem \ref{thm:main}.
Throughout this section, we fix a matching queue $(G,\lm,\Phi)_{\ssc}$, where $\Phi$ is of priority type, so that
$Q$ in (\ref{eq:defQ}) is a CTMC with state space 
\bequ \label{stateSpace}
\mathbb{G} := \Bigl\{z \in \ZZ^{|\maV|}_+ :\, z_iz_j=0,\,\mbox{ for any }i\in \maV\mbox{ and }j\in\maE(i)\Bigl\}.
\eeq

\subsection{Sample-Path Representation}

Before introducing the FWLLN for matching queues under priority policies, it is helpful to consider the
sample-path representation of the CTMC $Q$. To that end, note that
for each $i \in \maV$, $Q_i(t)$ increases by $1$ at time $t$ if
there is an arrival to node $i$ at $t$ and $Q_j(t) = 0$ for all $j
\in \maE(i)$; $Q_i(t)$ decreases by $1$ at time $t$ (when it is
positive) if there is an arrival to one of the neighbors $j \in
\maE(i)$, and all the buffers in $\maE(j)$ to which $j$ gives a
higher priority are empty. To express these dynamics, we introduce
the following subsets of $\mathbb G$: for any $i \in \maV$, we let
\bequ \label{eqAcceptSets} \bsplit
\maN_i & := \left\{z\in \mathbb G;\,z_i>0\right\}; \\
\maO_i & := \{z \in \mathbb G : z_j = 0 \qforallq j \in \maE(i)\};  \\
\maP_{j}(i) & := \{z \in \mathbb G : z_k = 0 \qforallq k \in \Phi_{j}(i)\},~ j\in\maE(i),
\end{split}
\eeq where $\Phi_{j}(i)$ is the list of all the neighbors of node
$j$ to which node $j$ gives a higher priority than to node $i$,
namely, \begin{equation} \label{eq:Phiji} \Phi_j(i) = \Bigl\{k \in
\maE(j);\,j\mbox{ gives priority to } k \mbox{ over } i\mbox{
according to }\Phi\Bigl\}.
\end{equation}

Let $\mathscr A$ denote the infinitesimal generator of the queue process $Q$.
Then, by the definition of the matching policy $\Phi$, the only
positive terms $\mathscr A(z,y)$,  for all $y,z\in \mathbb G$, are given by
\begin{equation}
\label{XnGenerator0}
\left\{\begin{array}{ll}
\mathscr{A}(z, z+\gre_{i}) &=\displaystyle\lm_i \ind_{\maO_i}(z),\,i\in \maV;\\\\
\mathscr{A}(z, z-\gre_{i}) &=\displaystyle\ind_{\maN_i}(z) \sum_{j
\in \maE(i)} \left(\lm_j \ind_{\maP_{j}(i)}(z)\right),\,i\in \maV,
\end{array}\right.
\end{equation}
where $\ind_A(\cdot)$ is the indicator function of the set $A$.
Consequently (see, e.g., \cite{PTW07}), for all $i \in \maV$,
we can represent the sample path of $Q_i$ using the independent Poisson arrival processes $N_i,\,i\in \maV$, via
\bequ \label{repForQn}
\bsplit
Q_i(t) & = Q_i(0) + \int_0^t \ind_{\maO_i}(Q(s-)) dN_i(s) \\
& \quad - \sum_{j\in \maE(i)}\int_0^t \ind_{\maN_i}(Q(s-))  \ind_{\maP_{j}(i)}(Q(s-))  dN_j \left(s \right), \quad t \ge 0,
\end{split}
\eeq where $Q(t-)$ denotes the left limit of $Q$ at the time point
$t$.

\begin{ex}
\label{ex:pendant}
For the pendant graph in which node $3$ prioritizes nodes $1$ and $2$ over $4$, as depicted in Figure \ref{Fig:pendant},
the subsets in \eqref{eqAcceptSets} become
\bes
\maO_4 = \left\{z \in \mathbb G : z_3 = 0 \right\} \qandq \maP_{3}(4) = \left\{z \in \mathbb G : z_1 = z_2 = 0\right\}.
\ees
The sample paths of $Q_4$ in that case can be represented via
\bes
Q_4(t) = Q_4(0) + \int_0^t \ind_{\maO_4}(Q(s-)) dN_4(s) - \int_0^t \ind_{\maN_4 \cap \maP_3(4)}(Q(s-)) dN_3(s),\,t\ge 0.
\ees
\end{ex}

\subsection{Marginal process corresponding to a particular node}\label{subsec:marginal}
The fluid limit we are about to introduce is formulated for a
particular class of models, exhibiting the following situation. When
fixing the arrival-rate vector $\lambda$ and the admissible matching
policy $\Phi$, and when isolating a single node for which we take
the corresponding initial buffer content to be strictly positive,
the content process of all the nodes different from that node and
its neighbors, coincides in law with an ergodic CTMC.

Formally, fix a matching queue
$\left(G,\lambda,\Phi\right)_{\ssc}$ and a node $i_0$ of $G$. Let
\[\maS=\maV\setminus(\{i_0\}\cup \maE(i_0)),\]
and index the elements of $\maS$ as follows,
\begin{equation}
\label{eq:defS}
\maS=\Bigl\{i_1,...,i_{|\maS|}\Bigl\}.
\end{equation}
Now, for any $z\in \mathbb G$ such that $z_{i_0}>0$,
the only positive terms $\mathscr A(z,y)$, $y\in \mathbb G$, are given by
\begin{equation}
\label{eq:genpremarginal}
\left\{\begin{array}{ll}
\mathscr{A}(z, z+\gre_{i_0}) &=\lm_{i_0}\,;\\\\
\mathscr{A}(z, z-\gre_{i_0}) &=\displaystyle\sum_{j \in \maE(i_0)} \left(\lm_j \ind_{\maP_{j}(i_0)}(z)\right)\,;\\\\
\mathscr{A}(z, z+\gre_{i_\ell}) & = \lm_{i_\ell} \ind_{\maO_{i_\ell}}(z),~ \ell\in \llbracket 1,|\maS| \rrbracket\,;\\\\
\mathscr{A}(z, z-\gre_{i_\ell}) & = \displaystyle\ind_{\maN_{i_\ell}}(z)
\sum_{\substack{j \in \maE(i_\ell):\\ i_0 \not\in \Phi_{j}\left(i_{\ell}\right)}} \left(\lm_j \ind_{\maP_{j}(i_\ell)}(z)\right), ~\ell\in \llbracket 1,|\maS| \rrbracket.
\end{array}\right.
\end{equation}
To see this, observe that, by the definition of $\Phi$,
$\ind_{\maP_{j}(i_\ell)}(z)=0$ for all $\ell$ and
$j \in \maE(i_\ell)$ such that $i_0 \in \Phi_{j}\left(i_{\ell}\right)$, since $z_{i_0}>0$.

Let $S=\left\{S(t)\,:\,t\ge 0\right\}$ denote the restriction of the process $Q$ to the nodes of $\maS$, i.e.,
\begin{equation}
\label{eq:defprocessS}
S= \left(S_1,S_2,...,S_{|\maS|}\right):=\left(Q_{i_{1}},...,Q_{i_{|\maS|}}\right).
\end{equation}
Then $S$ achieves values in the following subset $\mGS$ of
$\ZZ_{+}^{|\maS|}$,
\begin{equation}
\label{eq:defmGS}
\mGS=\Bigl\{x \in \ZZ_+^{|\maS|}\,:\,x_kx_\ell=0\mbox{ for }k,\ell \in \llbracket 1,|\maS| \rrbracket
\mbox{ such that }i_k \in \mathcal E\left(i_\ell\right)\Bigl\}.
\end{equation}

\noindent Analogously to (\ref{eqAcceptSets}), we define the following subsets of $\mGS$: For any $\ell \in \llbracket 0,|\maS| \rrbracket$,
\begin{align}
\maN^{\maS}_{i_\ell} & := \left\{x\in \mGS :\,x_\ell>0\right\}; \nonumber\\ 
\maO^{\maS}_{i_\ell} & := \Bigl\{x \in \mGS : x_j = 0 \qforallq j \in \llbracket 1,|\maS| \rrbracket \mbox{ such that }i_j\in \maE(i_\ell)\Bigl\}; \nonumber
\end{align}
and for any $\ell \in \llbracket 0,|\maS| \rrbracket$ and $j\in\maE(i_\ell)$, 
\bequ \label{eq:defmaPS}
\maP^{\maS}_{j}(i_\ell) := \Bigl\{x \in \mGS : x_k = 0
            \mbox{ for all } k \mbox{ such that }i_k\in \Phi_{j}(i_\ell)\Bigl\}. \,
\eeq

\begin{definition}
\label{def:marginal}
The {\em marginal process} corresponding to node $i_0$ is the $\mGS$-valued
CTMC $\chi := \{\chi(t) : t \ge 0\}$, whose infinitesimal generator $\mathscr{A}^{\maS}$ has the following positive terms,
\bequ \label{XnGenerator}
\left\{\begin{array}{ll}
\mathscr{A}^{\maS}(x, x+\gre_\ell) & = \lm_{i_\ell} \ind_{\maO^{\maS}_{i_\ell}}(x),~ \ell\in \llbracket 1,|\maS| \rrbracket\,;\\
&\\
\mathscr{A}^{\maS}(x, x-\gre_\ell) & =
\displaystyle\ind_{\maN^{\maS}_{i_\ell}}(x)
\sum_{\substack{j \in \maE(i_\ell):\\ i_0 \not\in \Phi_{j}\left(i_{\ell}\right)}} \left(\lm_j \ind_{\maP^{\maS}_{j}(i_\ell)}(x)\right), ~\ell\in \llbracket 1,|\maS| \rrbracket. \\
\end{array}\right.
\eeq
\end{definition}
Observe that for any $z \in \mathbb G$,  if $x \in \mGS$ is defined by $x=\left(z_{i_1},...,z_{i_{|\maS|}}\right)$, then for all $\ell \in \llbracket 1,|\maS| \rrbracket$ we have
\begin{align*}
\ind_{\maO_{i_\ell}}(z)&=\ind_{\maO^{\maS}_{i_\ell}}(x)\,;\\
\ind_{\maN_{i_\ell}}(z)&=\ind_{\maN^{\maS}_{i_\ell}}(x),
\end{align*}
and if $z_{i_0}>0$, as we have $z_j=0$ for all $j \in \maE(i_0)$, by definition of $\Phi$ we obtain
\[\ind_{\maP_{j}(i_\ell)}(z) = \ind_{\maP^{\maS}_{j}(i_\ell)}(x),\,\mbox{ for all }j \in \maE(i_\ell).\]
Therefore, in view of (\ref{eq:genpremarginal}) and
(\ref{XnGenerator}) we can provide a more intuitive definition of
the marginal process associated to the node $i_0$: it is a Markov
process on $\mGS$ which coincides in distribution with the
restriction $S$ of the process $Q$ to its coordinates in $\maS$,
conditionally on the $i_0$-th coordinate of $Q$ being positive.

\begin{ex}[Example \ref{ex:pendant}, continued]
\label{ex:2} Set $i_0=4$.
 Then we have $\maS=\{1,2\}$. Set $i_1=1$
and $i_2=2$. We thus have
\begin{equation}
\label{eq:defE2}
\mGS=\left(\{0\}\times \ZZ_+\right) \cup \left(\ZZ_+\times \{0\}\right)=:\mathbb E_2,
\end{equation}
and the following subsets of $\mathbb E_2$,
\begin{align*}
\maO^{\maS}_1 &= \ZZ_+\times \{0\}; \quad\quad \maP^{\maS}_{3}(1) = \maP^{\maS}_{2}(1) = \maN^{\maS}_1 = \ZZ_{++}\times \{0\};\\
\maO^{\maS}_2 &= \{0\}\times \ZZ_+; \quad\quad \maP^{\maS}_{3}(2) = \maP^{\maS}_{1}(2) = \maN^{\maS}_2 = \{0\}\times \ZZ_{++}.
\end{align*}
Thus the positive terms of the generator $\mathscr{A}^{\maS}$ are
\begin{equation}
\label{eq:transPendant0}
\left\{\begin{array}{ll}
\mathscr{A}^{\maS}(x,x+\gre_1)=\lambda_1 &\mbox{ for }x\in \ZZ_+\times \{0\};\\
\mathscr{A}^{\maS}(x,x-\gre_1)=\lambda_3+\lambda_2 &\mbox{ for }x\in \ZZ_{++}\times \{0\};\\
\mathscr{A}^{\maS}(x,x+\gre_2)=\lambda_2 &\mbox{ for }x \in \{0\}\times \ZZ_+;\\
\mathscr{A}^{\maS}(x,x-\gre_2)=\lambda_3+\lambda_1 &\mbox{ for }x\in \{0\}\times \ZZ_{++}.
\end{array}\right.
\end{equation}
\end{ex}

\subsection{The FWLLN} \label{secFWLLN}

Throughout this section, fix the matching queue
$\left(G,\lambda,\Phi\right)_{\ssc}$ and the node $i_0 \in \maV$.
We consider the sequence of fluid-scaled processes $\{\bar Q^n : n
\ge 1\}$, defined via
\bequ \label{Scaling}
\barq^n(t)=\frac{Q^n(t)}{n}:=\frac{Q(nt)}{n},\,t\ge 0, \quad n \ge 1.
\eeq
Similarly, recalling (\ref{eq:defprocessS}) we define
\begin{equation}
\label{eq:scalingS} \bar S^n(t)=\frac{S^n(t)}{n}:=\frac{S(nt)}{n},\,t\ge 0, \quad n \ge 1.
\end{equation}
For $n\in \Z_{++}$, we will use the notation $N^n_i(\cdot)$ for
the time-scaled Poisson arrival process to node $i$, namely,
$N^n_i(\cdot) = N_i(n \cdot)$, $i \in \maV$.
We also denote by $\chi^n$ the $n$-th marginal process corresponding to $i_0$, defined by
\begin{equation} \label{scalingChi}
\chi^n(t) = \chi(nt),\,t\ge 0,
\end{equation}
and define \begin{equation}\label{eq:scalingChi} \bar
\chi^n(t)=\frac{\chi^n(t)}{n},\,t\ge 0, \quad n \ge 1.
\end{equation}


The insufficiency of {\sc Ncond} to ensure the stability of a given matching queue will be shown via the following lemma.
\begin{lemma} \label{lmUnstable}
If there exists an initial condition $Q^n(0) \in \mathbb G$ such that $\barq^n \Ra \barq$ in $\D^{|\maV|}$ as $n\tinf$ and
$\maP : = \{i \in \maV : \barq_i(0) > 0\} \ne \emptyset$,
and if for some $i \in \maP$ it holds that $\barq_i$ is nondecreasing, then $Q$ is either transient or null recurrent.
In particular, the corresponding matching queue $(G,\lambda,\Phi)_{\ssc}$ is unstable.
\end{lemma}

\begin{proof}
By Proposition 9.9 in \cite{RobertBook}, if $Q$ is positive recurrent there exists a (possibly random) time $\maT$,
such that $\maT < \infty$ w.p.1 and $\barq(t) = 0$ for all $t \ge \maT$.
\end{proof}


For the fluid analysis, we make two assumptions. 
\begin{assumption} \label{AssInitial}
$Q^n(0) \in \mathbb G$ for any $n\ge 1$, and $\barq^n(0) \Ra \barq(0)$ as $n\tinf$, where $\barq(0)$ is a deterministic element of $\RR^{|\maV|}$,
with $\barq_{i_0}(0) > 0$ and $\barq_i(0) = 0$, $i \in \maV \setminus \{i_0\}$.
\end{assumption}

\begin{assumption} \label{AssErgo}
For all $n\ge 1$, the $\mGS$-valued process $\chi^n$ is ergodic with
stationary probability $\pi^n$.
\end{assumption}

For $n \ge 1$, let
\bequ \label{rho^n}
\rho^n := \rho^n(Q^n(0)) := \inf\{t \ge 0 : Q^n_{i_0}(t) = 0\},
\eeq
with $\inf \emptyset := \infty$.
\begin{lemma} \label{lmBndFluid}
Consider the sequence $\{\barq^n : n \in \ZZ_{++}\}$ corresponding to a system $(G,\lambda,\Phi)_{\ssc}$.
Then there exist $n_0 \in \ZZ_+$ and $\delta > 0$ such that $\rho^n(Q^n(0)) > \delta$ w.p.1 for all $n \ge n_0$.
In particular, there exists $n_0 < \infty$, such that
\bequ \label{QcPos}
\inf_{0 \le t < \delta}\barq^n_{i_0}(t) > 0 ~ w.p.1 \qforallq n \ge n_0.
\eeq
\end{lemma}

\begin{proof}
We use a simple coupling argument. Consider the matching queue
$(G,\lambda,\wt\Phi)_{\ssc}$ (with the same graph $G$ and arrival-rate vector $\lm$ as in the statement of the lemma),
where $\wt \Phi$ is the priority policy under which each $i \in \maE\left(i_0\right)$ gives the highest priority to node $i_0$.
If the corresponding queue process $\wt Q^n$ is given the same Poisson processes $\{N^n_i : i \in \maV\}$ of the original system,
we clearly have
\bes
Q^n_{i_0}(t) \ge \wt{Q}^n_{i_0}(t) := Q^n_{i_0}(0) + N_{i_0}(n t) - \sum_{j \in \maE\left(i_0\right)} N_j(n t), \quad 0 \le t \le \delta^n,
\ees
where $\delta^n := \inf\{t > 0 : \wt{Q}^n_{i_0}(t) = 0\}$.

Dividing $\wt{Q}^n_{i_0}$ by $n$ and taking $n \tinf$, we obtain from the functional strong law of large numbers (FSLLN) for the Poisson process, that
$\wt{Q}^n/n$ converges w.p.1 to $\wt{q}$, where
\bequ
\wt{q}_{i_0}(t) := \barq_{i_0}(0) + \left(\lm_{i_0} - \sum_{j \in \maE(i_0)} \lm_j\right) t, \quad 0 \le t < \delta,
\eeq
where
\bes
\delta := \frac{\barq_{i_0}(0)}{\sum_{j \in \maE(i_0)} \lm_j - \lm_{i_0}} \qifq \sum_{j \in \maE(i_0)} \lm_j - \lm_{i_0} > 0 \qandq \delta := \infty \mbox{ otherwise}.
\ees
The uniform convergence over compact subintervals of $[0, \delta)$ of the lower bounding processes $\wt{Q}^n_{i_0}/n$ to a strictly positive function gives \eqref{QcPos}.
\end{proof}

Before presenting the fluid limit, we explain the intuition behind the expression for $\barq_{i_0}$ that we obtain.
Since $\barq_{i_0}$ is strictly positive over an interval $[0, \delta)$ by Lemma \ref{lmBndFluid}, if $S^n(0) \deq \chi^n(0)$ for all $n \ge 1$,
then from (\ref{eq:genpremarginal}), \eqref{XnGenerator} describes the infinitesimal rates of $S^n$ over $[0, \delta)$, so that
\bequ \label{XiEqDistSn}
\{S^n(t) : 0 \le t < \delta\} \deq \{\chi^n(t) : 0 \le t < \delta\}.
\eeq
Hence, $S^n$ is locally (over $[0, \delta)$) distributed as a CTMC, which is ergodic by Assumption \ref{AssErgo}.
Thus it is not hard to show that $\bar S^n$ converges to $0$ over that interval; see the proof of Theorem \ref{thFWLLN}.
Nevertheless, the dynamics of $\bar S^n$ determine those of $\barq^n_{i_0}$ for each $n$, as is clear from \eqref{repForQn}, and the affect
of $\bar S^n$ on $\barq^n_{i_0}$ does not diminish as $n$ increases to infinity.
However, the ``small'' process $S^n$ is also ``fast'' relative to $\barq^n_{i_0}$, since
\eqref{scalingChi} implies that,
regardless of the distribution of $\chi^n(0)$, $\chi^n(t) \on{d}{\approx} \chi(\infty)$, for any $t > 0$ and for all large-enough $n$,
where $\chi(\infty)$
denotes a random variable having the stationary distribution of $\chi$.
(We write $\on{d}{\approx}$ if, in the limit as $n\tinf$, the distribution of $\chi^n(t)$
at time $t$ converges to the stationary distribution of $\chi$, i.e., $\chi^n(t) \Ra \chi(\infty)$ in $\mGS$.)
Then \eqref{XiEqDistSn} implies that $S^n(t)$ converges to $\chi(\infty)$ as well as $n\tinf$, $0 < t < \delta$.
Such a result is known in the queueing literature as a {\em pointwise stationarity}, e.g., \cite{Whitt-PSA}.
Of course, to obtain a FWLLN, the convergence must hold uniformly in $t$ over the interval $[0, \delta)$,
namely, the aforementioned stochastic AP must hold, but the intuition for the fast averaging phenomenon is similar.

Formally, let $\pi$ denote the stationary distribution of the CTMC $\chi$ whose generator
$\mathscr{A}^{\maS}$ is given in \eqref{XnGenerator}, i.e.,
\bequ \label{piDist}
\pi(\mathcal{Z}) = P(\chi(\infty) \in \mathcal{Z}), \quad \mathcal{Z} \subseteq \mGS.
\eeq

\begin{theorem}{$($FWLLN$)$} \label{thFWLLN}
Let $(G,\lambda,\Phi)_{\ssc}$ be a matching queue such that $\Phi$
is of the priority type. If, for some node $i_0$
\bequ
\label{CondDecrease} \lm_{i_0} - \sum_{j \in \maE(i_0)} \lm_j
\pi\left(\maP^{\maS}_j(i_0)\right) < 0, \eeq for $\pi$ in
\eqref{piDist} and $\maP^{\maS}_j(i_0)$, $j \in \maE(i_0)$ in
(\ref{eq:defmaPS}), then $\rho^n \Ra \rho$ in $\RR$ as $n\tinf$,
for $\rho^n$ in \eqref{rho^n}, where
\bequ \label{rho} \rho :=
\frac{\barq_{i_0}(0)}{\sum_{j \in \maE(i_0)}\lm_j
\pi\left(\maP^{\maS}_j(i_0)\right) - \lm_{i_0}}.
\eeq
Otherwise, $\rho^n \Ra \infty$. In either case, $\barq^n \Ra \barq$
in $\D^{|\maV|}[0, \rho)$ as $n \tinf$, where \bequ \label{FluidLim}
\bsplit
\barq_{i_0}(t) & = \barq_{i_0}(0) + \left(\lm_{i_0} - \sum_{j \in \maE(i_0)} \lm_j \pi\left(\maP^{\maS}_j(i_0)\right) \right) t, \\
\barq_i(t) & = 0, \quad i \in \maV \setminus \{i_0\}.
\end{split}
\eeq
\end{theorem}
The proof of Theorem \ref{thFWLLN} is given in Section \ref{secFluidProof}. From Theorem \ref{thFWLLN} and Lemma \ref{lmUnstable} it immediately follows that
\begin{corollary} \label{corNS}
If $\rho^n \Ra \infty$, for $\rho^n$ in \eqref{rho^n}, then $(G,\lambda,\Phi)_{\ssc}$ is unstable.
\end{corollary}

It is significant that we can compute the stationary probabilities $\pi(\cdot)$ in \eqref{piDist} in some cases, using reversibility arguments.

\section{The Pendant Graph and the 5-Cycle} \label{sec:Examples}

In this section we analyze matching queues defined on the pendant graph and the 5-cycle,
using the fluid limit in Theorem \ref{thFWLLN}.
In both cases, 
the stationary probability $\pi$ (on $\mathbb E_2$ in \eqref{eq:defE2}) can be computed explicitly,
so that the stability region of the corresponding matching queues can be fully characterized.

\subsection{The Pendant graph} \label{subsec:pendant}
We start with the model depicted in Figure \ref{Fig:pendant}.
\begin{proposition}
\label{pro:Pendant}
Let $G$ be the pendant graph and $\Phi$ the matching policy depicted in Figure \ref{Fig:pendant}.
Consider an arrival-rate vector $\lambda \in $ {\sc Ncond}$_{\ssc}(G)$, i.e.,
\begin{equation}
\label{eq:NcondPendant}
\lambda_4<\lambda_3<\lambda_4+\lambda_1+\lambda_2, \quad \lambda_4+\lambda_1<\lambda_3+\lambda_2 \qandq  \lambda_4+\lambda_2<\lambda_3+\lambda_1.
\end{equation}
If $\barq^n(0) \Rightarrow x \boe_4$ in $\RR^4$ for some $x \in \RR_{++}$, then $\barq^n \Ra \barq$ in $\D^4[0, \rho)$ as $n\tinf$,
where
\bequ \label{FluidPendant}
\barq(t) = \left(0,0,0,x+\left(\lambda_4 - \lm_3\alpha\right)t\right), \quad 0 \le t < \rho,
\eeq
for $\rho := x/(\lm_3\af-\lm_4)$ if $\af > \lm_4/\lm_3$ and $\rho := \infty$ otherwise, and for
\begin{equation}
\label{eq:defalphaPendant}
\alpha := \left[1+{\lambda_1 \over \lambda_3+\lambda_2-\lambda_1}+{\lambda_2 \over \lambda_3+\lambda_1-\lambda_2}\right]^{-1}
={\left(\lambda_3\right)^2-(\lambda_1-\lambda_2)^2 \over \lm_3(\lambda_3+\lambda_1+\lambda_2)}.
\end{equation}
\end{proposition}
\begin{proof}
The result follows from Theorem \ref{thFWLLN}. In the present case, we set $i_0=4$, so that
the marginal process $\chi$ is a $\mathbb E_2$-valued CTMC having the generator
$\mathscr A^{\maS}$ in (\ref{eq:transPendant0}) (see Example \ref{ex:2}).
For $\af$ in \eqref{eq:defalphaPendant}, let $\pi(0,0) = \af$ and
\[\pi(x)=\left\{\begin{array}{ll}
\alpha\left({\lambda_1 \over \lambda_3+\lambda_2}\right)^i&\mbox{ for }x=(i,0),\,i \ge 1;\\
\alpha\left({\lambda_2 \over \lambda_3+\lambda_1}\right)^j&\mbox{ for }x=(0,j),\,i \ge 1.
\end{array}
\right.
\]
It is easy to check that under (\ref{eq:NcondPendant}),
$\pi$ is a probability vector satisfying the detailed balance equations for $\chi$,
so that it is the unique stationary distribution of this reversible CTMC. In particular, Assumption \ref{AssErgo} holds.
Since items of class 3 give priority to $1$ and $2$ over $4$, we have that $\maP^{\maS}_3(4) = (0,0),$
so the stated convergence of $Q^n$ to the fluid limit in \eqref{FluidPendant} follows from (\ref{FluidLim}).
\end{proof}

\begin{proposition} \label{propPendant_iff}
The matching queue $(G,\lm,\Phi)_{\ssc}$ corresponding to the pendant graph $G$ and the
priority policy $\Phi$ represented in figure \ref{Fig:pendant} is
stable if and only if {\sc Ncond} holds together with
\begin{equation}
\label{eq:CNSpendant} \lm_4 < \af \lm_3,
\end{equation}
for $\af$ in \eqref{eq:defalphaPendant}.
\end{proposition}

\begin{proof}
The necessity of {\sc Ncond} has been shown in
Theorem \ref{thNcond}. Also, it follows from Proposition
\ref{pro:Pendant} that, for any initial condition of the form
$(0,0,0,x)$, $x > 0$, the fluid limit $\barq$ will hit the origin if
and only if $\lm_4 < \af\lm_3$. The necessity of
(\ref{eq:CNSpendant}) then follows from Lemma \ref{lmUnstable}.
To show sufficiency, we apply Dai's
result in \cite{DaiFluid}. To that end, we must consider {\em all
possible initial conditions} for the fluid limit, and show that the
origin is guaranteed to be hit in finite time.

First assume that $\barq_3(0) > 0$, so that all other queues are empty initially.
In that case, and as long as the class-$3$ queue is strictly positive, its drift down (toward $0$)
is $\lm_4 + \lm_1 + \lm_2$, which is larger than the drift up $\lm_3$ by (\ref{eq:NcondPendant}).
In particular, during the initial interval over which $\barq_3 > 0$,
the class-$3$ queue is distributed as an ergodic birth and death (BD) process whose fluid limit is
known to be ({\em e.g.}, Proposition 5.16 in \cite{RobertBook})
$$\barq_3(t) = \barq_3(0) + (\lm_3 - \lm_4 - \lm_1 - \lm_2)t, \quad 0\le t \le {\barq_3(0) \over \lm_4 + \lm_1 + \lm_2 - \lm_3},$$
so that the fluid queue hits the origin in finite time.

Now assume that $\barq_4(0) > 0$. Then at most one of $\barq_1(0)$
or $\barq_2(0)$ can be strictly positive. Say $\barq_1(0) > 0$. In
that case, the matching policy we consider implies that, as long as
$\barq_1 > 0$, all the arriving items of classes $3$ and $2$ are
matched with class-$1$ items. Hence, as long as the class-1 queue
process is positive, it is distributed as a BD process having a
constant birth rate $\lm_1$ and a constant death rate $\lm_3 +
\lm_2$. This BD process is ergodic due to (\ref{eq:NcondPendant}),
and its fluid limit is \bes \barq_1(t) = \barq_1(0) + (\lm_1 - \lm_3
- \lm_2)t, \quad 0 \le t \le {\barq_1(0) \over \lm_3+\lm_2-\lm_1}.
\ees In particular, the fluid process $\barq_1$ will hit $0$ in
finite time, so that $\barq$ will hit the origin in finite time by
Proposition \ref{pro:Pendant}. A similar argument applies when $\barq_2(0) > 0$.

Now, since the prelimit processes $Q_i$, $i = 1,2,3$ have drifts towards $0$ whenever any of them is strictly positive,
the fluid limit must remain in state $0$ after hitting this state, and Proposition \ref{pro:Pendant} shows that $\barq_4$ will also remain fixed at $0$ after hitting that state.
Thus, the ergodicity of the system follows from Theorem 4.2 in \cite{DaiFluid}.
\end{proof}

\begin{remark}
\label{remDiscrete} {\em A discrete version of Proposition
\ref{propPendant_iff} was proved in \cite{MaiMoyal15} for the
symmetric model $\left(G,\mu,\Phi\right)_{\sd}$ in which $\mu_\lm(1) =
\mu_\lm(2)$ (so that a lower-dimensional process can be considered), via
subtle Lyapunov-stability arguments. Plugging $\lm_1 = \lm_2$ in
\eqref{eq:defalphaPendant} and recalling (\ref{eq:defmucont}),
Proposition \ref{propPendant_iff} gives that the discrete model
$\left(G,\mu_\lm,\Phi\right)_{\sd}$ corresponding to Figure
\ref{Fig:pendant} is stable if and only if
\begin{equation}\label{eq:NcondPendantSym}
\mu_{\lm}(4) < \mu_{\lm}(3) < \mu_{\lm}(4)+2\mu_{\lm}(1)\,\mbox{ and
}\,\left(\mu_{\lm}(3)\right)^2 >
\mu_{\lm}(4)(1-\mu_{\lm}(4)).\end{equation} As
$\mu_{\lm}(2)=\mu_{\lm}(1)$, the left-hand condition above is
equivalent to
\[\left\{\begin{array}{ll}
\mu_{\lm}(3) &<\mu_{\lm}(4)+\mu_{\lm}(1)+\mu_{\lm}(2) =\mu_{\lm}\left(\maE(3)\right);\\
\mu_{\lm}(4) &< \mu_{\lm}(3)=\mu_{\lm}\left(\maE(4)\right)\\
\mu_{\lm}(1) &< \mu_{\lm}(3)+\mu_{\lm}(1)=\mu_{\lm}(3)+\mu_{\lm}(2)=\mu_{\lm}\left(\maE(1)\right)\\
\mu_{\lm}(2) &< \mu_{\lm}(3)+\mu_{\lm}(2)=\mu_{\lm}(3)+\mu_{\lm}(1)=\mu\left(\maE(2)\right)\\
\mu_{\lm}\left(\{1,4\}\right) &= \mu_{\lm}(1) + \mu_{\lm}(4)
<\mu_{\lm}(2)+\mu_{\lm}(3) =
\mu_{\lm}\left(\maE\left(\{1,4\}\right)\right);\\
\mu_{\lm}\left(\{2,4\}\right) &= \mu_{\lm}(2) + \mu_{\lm}(4)
<\mu_{\lm}(1)+\mu_{\lm}(3) =
\mu_{\lm}\left(\maE\left(\{2,4\}\right)\right).
\end{array}\right.\]
Thus, the measure $\mu_{\lm}$ satisfies the condition {\sc Ncond}($G$)
in p.5 of \cite{MaiMoyal15}. It is easy to see that
the second condition in (\ref{eq:NcondPendantSym}) is equivalent
to the right-hand condition defining the region denoted {\sc Stab}$(A)$ in Proposition 3 of \cite{MaiMoyal15} (after re-indexing
the nodes according to Figure 1 in \cite{MaiMoyal15}).
In particular, a measure $\mu_\lm$ with $\mu_\lm(1) = \mu_\lm(2)$ satisfies (\ref{eq:NcondPendantSym}) if and only if
it belongs to {\sc Stab}$(A)$, so we retrieved the stability
condition that was established in Proposition 3 of \cite{MaiMoyal15} for that particular case. (Note however, that we do not require $\mu_\lm(1) = \mu_\lm(2)$.)}
\end{remark}

As the following shows, the stability region of the model in Figure
\ref{Fig:pendant}, namely
$$\text{{\sc Ncond}}_{\ssc}(G) \cap \left\{\lambda \mbox{ satisfying (\ref{eq:CNSpendant}) }\right\},$$
is strictly contained in $\text{{\sc Ncond}}_{\ssc}(G)$.
\begin{proposition} \label{prop:nonsufficientpendant}
We have the strict inclusion
\[\left\{\lambda \mbox{ satisfying (\ref{eq:CNSpendant}) }\right\}\,\bigcap\, \mbox{{\sc Ncond}}_{\ssc}(G) \subsetneq \mbox{{\sc Ncond}}_{\ssc}(G).\]
\end{proposition}

\begin{proof}
Fix $\epsilon \in \left(0,{2/ 5}\right]$ and set
\[\left\{\begin{array}{ll}
\lambda_1&=\lambda_2={\epsilon / 2};\\
\lambda_3&={1\over 2}-{\epsilon / 4};\\
\lambda_4&={1\over 2}-{3\epsilon / 4}.
\end{array}\right.\]
It is then easily checked that $\lambda \in \mbox{{\sc Ncond}}_{\ssc}(G)$. However, simple algebra shows that
\begin{equation*}
\lambda_4 - \alpha\lambda_3 ={1\over 2}-{3\epsilon \over 4}-\left({1\over 2} - {\epsilon \over 4}\right){{1/2}-{\epsilon / 4} \over
{1/2}+{3\epsilon / 4}}= {{\epsilon / 4}\left(1-{5\epsilon / 2}\right) \over {1/ 2}+{3\epsilon / 4}} \ge 0,
\end{equation*}
so that $\lambda$ does not satisfy (\ref{eq:CNSpendant}).
\end{proof}

\subsection{The 5-cycle} \label{subsec:5cycle}
We now consider the matching queue corresponding to the $5$-cycle,
under the priority policy depicted in Figure \ref{Fig:5cycle}: nodes $1$ and $2$ prioritize each other, node $3$ gives priority to node $1$,
and node $4$ gives priority to node $2$.
For concreteness, we assume that node $5$ gives priority to node $4$.

As for the pendant graph analyzed above, the stability region of the
$5$-dimensional CTMC $Q$ is challenging, even in symmetric cases;
However, stability analysis is made considerably more simple {\em
via} the fluid limits analysis. In particular, we now obtain a
necessary and sufficient condition for stability of the matching
queue corresponding to the $5$-cycle and the matching policy $\Phi$
specified above, so that the stability region of the model is fully
characterized.

\begin{figure}[h!]
\begin{center}
\begin{tikzpicture}
\fill (1,2) circle (1.5pt) node[above]{5};
\fill (0,1) circle (1.5pt) node[left] {3};
\fill (2,1) circle (1.5pt) node[right] {4};
\fill (0,0) circle (1.5pt) node[below left] {1};
\fill (2,0) circle (1.5pt) node[below right] {2};
\draw[-] (0,0) -- (2,0);
\draw[-] (0,0) -- (0,1);
\draw[->,thick] (0,0.6) -- (0,0.4);
\draw[->,thick] (1.4,1.6) -- (1.6,1.4);
\draw[-] (2,0) -- (2,1);
\draw[->,thick] (2,0.6) -- (2,0.4);
\draw[-] (0,1) -- (1,2);
\draw[-] (2,1) -- (1,2);
\draw[->,thick] (1.8,0) -- (1.6,0);
\draw[->,thick] (0.2,0) -- (0.4,0);
\end{tikzpicture}
\caption[smallcaption]{The 5-cycle; arrows indicate priorities.}\label{Fig:5cycle}
\end{center}
\end{figure}
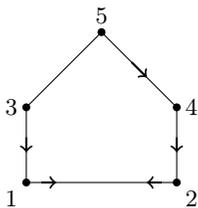



\noindent Let
\bequ \label{a1a2}
a := {\lambda_3(\lambda_1+\lambda_4) \over \lambda_1+\lambda_4-\lambda_2} + {\lambda_4 (\lambda_2+\lambda_3) \over \lambda_2+\lambda_3-\lambda_1}.
\eeq
\begin{proposition}
\label{pro:5cycle} Let $G$ be the 5-cycle, $\Phi$ be the priority
matching policy depicted in Figure \ref{Fig:5cycle}, and consider
$\lambda \in ${\sc Ncond}$_{\ssc}(G)$. Assume that for some $x \in
\RR_{++}$, $\barq^n(0) \Rightarrow x \,\boe_5$ in $\RR^5$. Then,
$\barq^n \Rightarrow \barq$ in $\D^5[0, \tilde{\rho})$ as $n\tinf$,
where, for $a$ in \eqref{a1a2},
\begin{align*}
\barq(t) & = \Bigl(0, 0, 0, 0, x + \left(\lambda_5 - a \tilde\alpha \right)t\Bigl), \quad 0 \le t < \tilde\rho;\\
\end{align*}
for 
\begin{equation} \label{eq:defalpha5cycle}
\tilde\alpha=\left[1+{\lambda_1 \over \lambda_2+\lambda_3-\lambda_1}+{\lambda_2 \over \lambda_1+\lambda_4-\lambda_2}\right]^{-1},
\end{equation}
and
$$\tilde\rho = {x \over a \tilde\af - \lm_5} \qifq \lm_5<a\tilde\af  \qandq \tilde\rho := \infty \text{ otherwise}.$$
\end{proposition}
It will be seen in the proof of Proposition \ref{pro:5cycle} that $\tilde \af$ is the normalizing constant that makes the solution
to the detailed balance equations a probability vector.
Note that {\sc Ncond}$_{\ssc}(G)$ implies that $a$ in \eqref{a1a2} and $\tilde\af$ in \eqref{eq:defalpha5cycle} are well-defined.
To see this, consider the independent sets $\{1\}$ and $\{2\}$, whose neighboring sets are $\maE(1) = \{2,3\}$ and $\maE(2) = \{1,4\}$, respectively.

\begin{proof}
Set $i_0=5$. In this case, from (\ref{XnGenerator}) the generator of
the associated $\mathbb E_2$-valued marginal process $\chi$
has the following positive terms,
\begin{equation}
\label{eq:trans5cycle}
\left\{\begin{array}{ll}
\tilde{\mathscr{A}}^{\maS}(x,x+\gre_1)=\lambda_1 &\mbox{ for }x\in \ZZ_+\times \{0\};\\
\tilde{\mathscr{A}}^{\maS}(x,x-\gre_1)=\lambda_3+\lambda_2 &\mbox{ for }x\in \ZZ_{++}\times \{0\};\\
\tilde{\mathscr{A}}^{\maS}(x,x+\gre_2)=\lambda_2 &\mbox{ for }x \in \{0\}\times \ZZ_+;\\
\tilde{\mathscr{A}}^{\maS}(x,x-\gre_2)=\lambda_1+\lambda_4 &\mbox{ for }x\in \{0\}\times \ZZ_{++}.
\end{array}\right.
\end{equation}
As in the proof of Proposition \ref{pro:Pendant}, one can easily check that under {\sc Ncond}$_{\ssc}(G)$,
\begin{equation}
\label{eq:defpitilde}\tilde \pi(x) = \left\{\begin{array}{ll}
\tilde\alpha\left({\lambda_1 \over \lambda_3+\lambda_2}\right)^i&\mbox{ for }x=(i,0),\,i\in\ZZ_+;\\
\tilde\alpha\left({\lambda_2 \over \lambda_1+\lambda_4}\right)^j&\mbox{ for }x=(0,j),\,i\in\ZZ_+
\end{array}
\right.
\end{equation}
is 
the unique stationary distribution of $\chi$. Since
\[\maP^{\maS}_3(5)=\{0\}\times\ZZ_+\mbox{ and }\maP^{\maS}_4(5)=\ZZ_+\times\{0\},\]
the statement follows from Theorem \ref{thFWLLN}.
\end{proof}

\begin{corollary} \label{cor:5cycle}
A necessary condition for the matching queue in Figure \ref{Fig:5cycle} to be stable is
\bequ \label{necessary5cycle}
\lm_5 < a \tilde\af, \quad \text{for $a$ in \eqref{a1a2}}.
\eeq
\end{corollary}


We next present a sufficient condition for the stability of the matching queue corresponding to Figure \ref{Fig:5cycle}.
Unlike the pendant graph, that sufficient condition is strictly stronger than the necessary condition of Corollary \ref{cor:5cycle}.
First note that nodes $1$ and $2$ are always stable, for any
$\lambda$ in {\sc Ncond}$_{\ssc}(G)$. This can be easily seen by observing that the drift down of the class $1$ queue process,
whenever the latter is positive, equals $\lm_2+\lm_3$,  which is
strictly less than the drift up $\lm_1$. Similarly, the downward
drift to $0$ of the class $2$ queue process is $\lm_1 + \lm_4$,
which is less than the upward drift $\lm_2$. It remains to consider
nodes $3$ and $4$.

\paragraph{Stability of Node 3}
If $\barq_3(0) > 0$, then $\barq_3$ is strictly positive over an interval $I \subset [0, \infty)$.
Over this interval $I$, the class-$2$ and class-$4$ queue processes behave as a fast-time-scale CTMC.
Just as Proposition \ref{pro:5cycle}, from Theorem \ref{thFWLLN} we obtain that, if $\barq^n(0) \Ra \barq_3(0)$ and $\barq^n_i(0) \Ra 0$, $i \ne 3$ in $\RR$ as $n\tinf$,
then $\barq^n \Ra \barq$ in $\D^5[0, \rho_3)$ as $n\tinf$, where
\bes
\left\{\begin{array}{ll}
\barq_3(t) & = \barq_3(0) + \Bigl[\lm_3 - c\af_{(2,4)}\Bigl]t, \quad 0 \le t < \rho_3, \\
   & \\
\barq_i(t) & = 0, \quad 0 \le t < \rho_3, \quad i \in \llbracket 1,5 \rrbracket \setminus \{3\},
\end{array}\right.
\ees
where
\bes 
\bsplit
\af_{(2,4)} & := \left(1 + {\lm_2 \over \lm_1 + \lm_4 - \lm_2} + {\lm_4 \over \lm_5+\lm_2-\lm_4}\right)^{-1}, \\
c& := {\lm_5(\lm_1+\lm_4) \over \lm_1 + \lm_4 - \lm_2} + {\lm_1(\lm_5+\lm_2) \over \lm_5+\lm_2-\lm_4}, \\
\rho_3 & := {\barq_3(0) \over (c_0+c_1)\af_{(2,4)} - \lm_3} \qifq \af_{(2,4)} > {\lm_3 \over c_0+c_1} \qandq \rho_3 := \infty, \text{ otherwise.}
\end{split}
\ees
Hence, in addition to requiring that {\sc Ncond}$_{\ssc}(G)$ holds, we must have $\rho_3 < \infty$, so that the fluid limit $\barq_3$ reaches $0$ in finite time.
In particular, a necessary condition for the stability of the model is that
\bequ \label{af24}
\lm_3 < c \af_{(2,4)},
\eeq
which is not implied by the necessary condition in Corollary \ref{cor:5cycle}.

\paragraph{Stability of Node 4}

Since node $5$ gives priority to class $4$ over class $3$, the
instantaneous downward drift of the class $4$ queue process, at any
time $t$ in which it is strictly positive, is $\lm_5 + \lm_2
\ind_{\{Q_1(t) = 0\}}$, while its upward drift is the constant
$\lm_4$. Then, Theorem \ref{thFWLLN} implies again that, if
$\barq^n_4(0) \Ra \barq_4(0)$, for some $\barq_4(0) > 0$, and
$\barq^n_i(0) \Ra 0$ in $\RR$ as $n\tinf$, $i \ne 4$, then over some
interval $I \subset [0, \infty)$ it holds that $\barq^n \Ra \barq$
in $\D^5[0, \rho_4)$, where \bes \left\{\begin{array}{ll}
\barq_4(t) & = \barq_4(0) + \left(\lm_4 - \lm_5 - \lm_2 \af_{(1,3)} \right) t, \quad 0 \le t < \rho_4, \\
& \\
\barq_i(t) & = 0, \quad 0 \le t < \rho_4, \quad i \in \llbracket 1,5 \rrbracket\setminus\{4\},
\end{array}
\right. \ees for \bes \af_{(1,3)} : = 1-{\lm_1 \over \lm_2 + \lm_3},
\ees and \bes \rho_4 := {\barq_4(0) \over \lm_5 + \lm_2 \af_{(1,3)}
- \lm_4} ~\text{ if } \af_{(1,3)} > (\lm_4 - \lm_5)/\lm_2 \qandq
\rho_4 := \infty \text{ otherwise.} \ees Therefore, in order for
node $4$ to be stable, we require that, in addition to having {\sc
Ncond}$_{\ssc}(G)$ hold, $\rho_4 < \infty$, or equivalently,
\bequ
\label{af13}
\lm_4 < \lm_2 \af_{(1,3)} +\lm_5.
\eeq

\paragraph{The Stability Region of the model of Figure \ref{Fig:5cycle}}

Similarly to the proof of Proposition \ref{propPendant_iff}, one can
show that, if each of the necessary conditions for stability of each
of the nodes holds, then the model is stable. In particular,
\begin{proposition} \label{th5cycle_iff}
The model $(G,\lm,\Phi)_{\ssc}$ corresponding to the $5$-cycle and the matching policy in Figure \ref{Fig:5cycle}
is stable if and only if $\lambda \in \mbox{{\sc Ncond}}_{\ssc}(G)$ and
all three inequalities \eqref{necessary5cycle}, \eqref{af24} and \eqref{af13} hold.
\end{proposition}
Similarly to the pendant graph (see Proposition \ref{prop:nonsufficientpendant}), we can check that the stability region
of the model is strictly included in $\mbox{{\sc Ncond}}_{\ssc}(G)$. Specifically,
\begin{proposition}
\label{prop:nonsufficient5cycle}
We have the strict inclusion
\[\left\{\lambda \mbox{ satisfying (\ref{necessary5cycle}) }\right\}\,\bigcap\, \mbox{{\sc Ncond}}_{\ssc}(G) \subsetneq \mbox{{\sc Ncond}}_{\ssc}(G).\]
\end{proposition}
\begin{proof}
Fix $\epsilon \in \left(0,{2/ 9}\right]$ and set
\[\left\{\begin{array}{ll}
\lambda_1&=\lambda_2={\epsilon / 2};\\
\lambda_3&=\lambda_4={1/ 4}-{\epsilon /8};\\
\lambda_5&={1/2}-{3\epsilon /4}.
\end{array}\right.\]
Clearly, $\lambda \in \mbox{{\sc Ncond}}_{\ssc}(G)$, but (\ref{necessary5cycle}) does not hold since
\begin{equation*}
\lambda_5 - a\tilde\alpha ={1\over 2}-{3\epsilon \over 4}-\left({1\over 2} -
{\epsilon \over 4}\right){{1/ 4}+{3\epsilon / 8} \over {1/ 4}+{7\epsilon / 8}}= {{\epsilon / 8}\left(1-{9\epsilon / 2}\right) \over
{1/ 4}+{7\epsilon / 8}}\ge 0.
\end{equation*}\qedhere
\end{proof}

\section{Proof of The Main Result} \label{secProofMain}

In this section we prove Theorem \ref{thm:main}, after introducing several key auxiliary results.

\subsection{Graphs Induced By Separable Graphs} \label{secInduced}

We start by proving the connection between separable graphs and the graphs investigated in Section
\ref{sec:Examples}.

\begin{lemma} \label{Lm:BipSepCycles}
For any connected graph $G$,
\begin{itemize}
\item[(i)] If $G$ is non-bipartite and non-separable, then it induces a pendant graph or an odd cycle of size $5$ or more. \label{Item1}
\item[(ii)] If $G$ is separable, then it does not induce a pendant graph, nor any odd cycle of size $5$ or more.
\end{itemize}
\end{lemma}

\begin{proof}
{\bf Proof of $(i)$. ~} Let $G$ be a non-bipartite and non-separable graph.
It is a classical result of graph theory (see {\em e.g.} Theorem 13.2.1 of \cite{LPV})
that $G$ contains an odd cycle $\breve G$ as a subgraph.
We consider two cases separately: $\breve G$ is a triangle, or $\breve G$ is of size $5$ or more.

$\underline{\mbox{{\bf Case 1}}}$. Assume that $\breve G$ is a triangle.
It is not possible that $\breve G=G$, because then $G$ would be a complete graph, and thereby a separable graph.
Thus $G$ has other nodes, in addition to the three ones of $\breve G$.

We prove the claim by induction. Assume that G induces a graph $\bar
G$ that consists of a separable graph $\check G$ of order $q \ge 3$
(the base case of the induction being $\check G=\breve G$),
connected to another node, which we denote as node $i$. We have the
following four alternatives, which are depicted in an example in
Figure \ref{Fig:Completion}.
\begin{enumerate}
\item[(a)] There may exist two independent sets $\maI_1$ and $\maI_2$ of $\check G$, and two nodes $i_1 \in \maI_{1}$ and $i_2 \in \maI_{2}$, such that $i \pv i_1$ and $i \pv i_2$.
In that case, as $\check G$ is of order $3$ or more, there exist an
independent set $\maI_3$, which is different from $\maI_1$ and
$\maI_2$, and a node $i_3 \in \maI_{3}$, such that $i \v i_3$. Then
$\check G$ induces the pendant graph $i \v i_3 \v i_1 \v i_2$, where
$i$ is only adjacent to $i_3$ and $(i_1,i_2,i_3)$ form a triangle.
In particular, $G$ induces a pendant graph.
\item[(b)] There may exist a maximal independent set $\maI_1$ of $\check G$ such that $i \v k$
for any $k  \in \check G\setminus \maI_{1}$, and $i \pv i_1$ for any
$i_1 \in \maI_{1}$. Then $\bar G$ is again a separable graph of
order $q$, having the same maximal independent sets as $\check G$,
except that $\maI_{1}$ is replaced by $\maI_{1}\cup \{i\}$.
\item[(c)] There may exist a maximal independent set $\maI_1$ of $\check G$ such that $i \v k$ for any $k  \in \check G\setminus \maI_{1}$, and $i \pv i_1$
for some (not necessarily unique) $i_1 \in \maI_1$ and $i \v j_1$ for some (again, not necessarily unique) $j_1 \in \maI_1$. In that case, $\bar G$
induces the pendant graph $i_1 \v i_2 \v j_1 \v i$, where
$i_1$ is only adjacent to $i_2$ and $(i_2,j_1,i)$ form a triangle.
\item[(d)] We may have $i \v j$ for any node $j \in \check G$. Then $\bar G$ is a separable graph of order $q+1$ whose maximal independent sets are those
of $\check G$, plus the independent set $\maI_{q+1}:=\{i\}$.
\end{enumerate}
To summarize, in cases (a) and (c), $\bar G$ and therefore $G$,
induce a pendant graph. In cases (b) and (d) $\bar G$ is separable,
and we cannot have $G=\bar G$. Thus there exists another node in $G$
that is connected to $\bar G$, and we can re-iterate the same
procedure for $\check G \equiv \bar G$. Eventually, some $\bar G$
induced in $G$ will exhibit either case (a) or (c), otherwise $G$
would be separable. This concludes the proof in this first case.

\begin{figure}[h!]
\begin{center}
\begin{tikzpicture}
\draw[-,very thin] (1,1) -- (2,-0.2);
\draw[-,very thin] (1,1) -- (3,0);
\draw[-,very thin] (1,1) -- (2,1.2);
\draw[-,very thin] (1,1) -- (3,1);
\draw[-,very thick] (2,1.2) -- (1,0);
\draw[-,very thin] (2,1.2) -- (3,0);
\draw[-,very thick] (2,1.2) -- (3,1);
\draw[-,very thick] (3,1) -- (1,0);
\draw[-,very thin] (3,1) -- (2,-0.2);
\draw[-,very thin] (1,0) -- (2,-0.2);
\draw[-,very thin] (1,0) -- (3,0);
\draw[-,very thin] (2,-0.2) -- (3,0);
\draw[-,very thin] (0,0.5) -- (1,1);
\draw[-,very thick]  (0,0.5) -- (3,1);
\draw[-,very thin] (0,0.5) -- (2,-0.2);
\draw[-,very thin] (0,0.5) -- (3,0);
\fill (0,0.5) circle (2pt) node[left] {\small{$i$}} ;
\fill (1,1) circle (2pt) node[above] {\small{1}} ;
\fill (2,1.2) circle (2pt) node[above] {\small{$2:=i_2$}} ;
\fill (3,1) circle (2pt) node[above right] {\small{$3:=i_3$}} ;
\fill (1,0) circle (2pt) node[below] {\small{$4:=i_1$}} ;
\fill (2,-0.2) circle (2pt) node[below] {\small{5}} ;
\fill (3,0) circle (2pt) node[below] {\small{6}} ;
\fill (1.5,-1) node[] {\scriptsize{case (a)}} ;
\draw[-,very thin] (7,1) -- (8,-0.2);
\draw[-,very thin] (7,1) -- (9,0);
\draw[-,very thin] (7,1) -- (8,1.2);
\draw[-,very thin] (7,1) -- (9,1);
\draw[-,very thin] (8,1.2) -- (9,0);
\draw[-,very thin] (8,1.2) -- (9,1);
\draw[-,very thin] (9,1) -- (7,0);
\draw[-,very thin] (9,1) -- (8,-0.2);
\draw[-,very thin] (7,0) -- (8,-0.2);
\draw[-,very thin] (7,0) -- (9,0);
\draw[-,very thin] (8,-0.2) -- (9,0);
\draw[-] (6,0.5) -- (8,1.2);
\draw[-,very thin]  (6,0.5) -- (9,1);
\draw[-,very thin] (6,0.5) -- (8,-0.2);
\draw[-,very thin] (6,0.5) -- (9,0);
\fill (6,0.5) circle (2pt) node[left] {\small{$i$}} ;
\fill (7,1) circle (2pt) node[above] {\small{1}} ;
\fill (8,1.2) circle (2pt) node[above right] {\small{2}} ;
\fill (9,1) circle (2pt) node[above] {\small{3}} ;
\fill (7,0) circle (2pt) node[below] {\small{4}} ;
\fill (8,-0.2) circle (2pt) node[below] {\small{5}} ;
\fill (9,0) circle (2pt) node[below] {\small{6}} ;
\fill (7.5,-1) node[] {\scriptsize{case (b): $\maI_1=\{1,4,i\}$}} ;
\draw[-,very thin] (1,-2) -- (2,-3.2);
\draw[-,very thin] (1,-2) -- (3,-3);
\draw[-,very thick] (1,-2) -- (2,-1.8);
\draw[-,very thin] (1,-2) -- (3,-2);
\draw[-,very thick] (2,-1.8) -- (1,-3);
\draw[-,very thin] (2,-1.8) -- (3,-3);
\draw[-,very thin] (2,-1.8) -- (3,-2);
\draw[-,very thin] (3,-2) -- (1,-3);
\draw[-,very thin] (3,-2) -- (2,-3.2);
\draw[-,very thin] (1,-3) -- (2,-3.2);
\draw[-,very thin] (1,-3) -- (3,-3);
\draw[-,very thin] (2,-3.2) -- (3,-3);
\draw[-,very thick] (0,-2.5) -- (1,-3);
\draw[-, very thick] (0,-2.5) -- (2,-1.8);
\draw[-,very thin]  (0,-2.5) -- (3,-2);
\draw[-,very thin] (0,-2.5) -- (2,-3.2);
\draw[-,very thin] (0,-2.5) -- (3,-3);
\fill (0,-2.5) circle (2pt) node[left] {\small{$i$}} ;
\fill (1,-2) circle (2pt) node[above left] {\small{$1:=i_1$}} ;
\fill (2,-1.8) circle (2pt) node[above] {\small{$2:=i_2$}} ;
\fill (3,-2) circle (2pt) node[above] {\small{3}} ;
\fill (1,-3) circle (2pt) node[below] {\small{$4:=j_1$}} ;
\fill (2,-3.2) circle (2pt) node[below] {\small{5}} ;
\fill (3,-3) circle (2pt) node[below] {\small{6}} ;
\fill (1.5,-4) node[] {\scriptsize{case (c)}} ;
\draw[-,very thin] (7,-2) -- (8,-3.2);
\draw[-,very thin] (7,-2) -- (9,-3);
\draw[-,very thin] (7,-2) -- (8,-1.8);
\draw[-,very thin] (7,-2) -- (9,-2);
\draw[-,very thin] (8,-1.8) -- (7,-3);
\draw[-,very thin] (8,-1.8) -- (9,-3);
\draw[-,very thin] (8,-1.8) -- (9,-2);
\draw[-,very thin] (9,-2) -- (7,-3);
\draw[-,very thin] (9,-2) -- (8,-3.2);
\draw[-,very thin] (7,-3) -- (8,-3.2);
\draw[-,very thin] (7,-3) -- (9,-3);
\draw[-,very thin] (8,-3.2) -- (9,-3);
\draw[-,very thin] (6,-2.5) -- (7,-2);
\draw[-,very thin] (6,-2.5) -- (7,-3);
\draw[-, very thin] (6,-2.5) -- (8,-1.8);
\draw[-,very thin]  (6,-2.5) -- (9,-2);
\draw[-,very thin] (6,-2.5) -- (8,-3.2);
\draw[-,very thin] (6,-2.5) -- (9,-3);
\fill (6,-2.5) circle (2pt) node[left] {\small{$i$}} ;
\fill (7,-2) circle (2pt) node[above] {\small{1}} ;
\fill (8,-1.8) circle (2pt) node[above right] {\small{2}} ;
\fill (9,-2) circle (2pt) node[above] {\small{3}} ;
\fill (7,-3) circle (2pt) node[below] {\small{4}} ;
\fill (8,-3.2) circle (2pt) node[below] {\small{5}} ;
\fill (9,-3) circle (2pt) node[below] {\small{6}} ;
\fill (7.5,-4) node[] {\scriptsize{case (d): $\maI_4=\{i\}$}} ;
\end{tikzpicture}
\caption[smallcaption]{Completing a separable graph with one node always yields to a separable graph ((b) and (d)),
or a non-separable graph inducing a pendant graph ((a) and (c)).}
\label{Fig:Completion}
\end{center}
\end{figure}
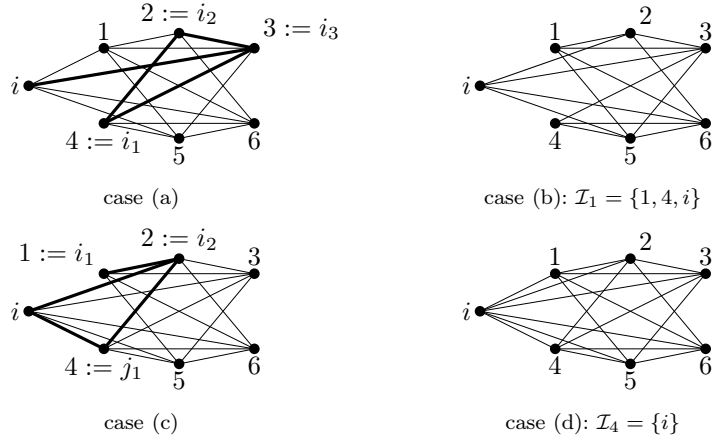

$\underline{\mbox{{\bf Case 2}}}$. Now assume that $\breve G$ is of length $5$ or more.
Assume that $G$ does not induce $\breve G$.
Therefore $G$ admits an edge $(i,j)$, where $i$ and $j$ are two nodes of $\breve G$. But drawing an edge inside an odd cycle always creates an odd cycle and an even cycle (see Figure \ref{Fig:cutOddcycle}).

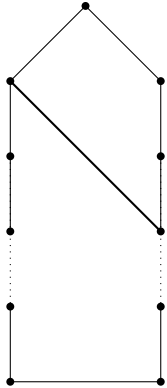
\begin{figure}[h!]
\begin{center}
\begin{tikzpicture}
\fill (1,2) circle (1.5pt);
\fill (0,1) circle (1.5pt);
\draw [-,thick] (0,1) -- (2,-1);
\fill (2,1) circle (1.5pt);
\fill (0,0) circle (1.5pt);
\fill (2,0) circle (1.5pt);
\fill (0,-1) circle (1.5pt);
\fill (2,-1) circle (1.5pt);
\draw[-, dotted] (0,0) -- (0,-2);
\fill (0,-2) circle (1.5pt);
\draw[-] (0,-2) -- (0,-3);
\fill (0,-3) circle (1.5pt);
\draw[-, dotted] (2,0) -- (2,-2);
\fill (2,-2) circle (1.5pt);
\draw[-] (2,-2) -- (2,-3);
\fill (2,-3) circle (1.5pt);
\draw[-] (1,2) -- (2,1);
\draw[-] (0,0) -- (0,1);
\draw[-] (2,0) -- (2,1);
\draw[-] (0,1) -- (1,2);
\draw[-] (0,0) -- (0,-1);
\draw[-] (2,-1) -- (2,0);
\draw[-] (0,-3) -- (2,-3);
\end{tikzpicture}
\caption[smallcaption]{Drawing an edge between two nodes of an odd cycle creates an odd cycle.}\label{Fig:cutOddcycle}
\end{center}
\end{figure}
By induction on the added edges, $\breve G$ finally induces an odd cycle of length $2p+1 \ge 3$. If $2p+1 \ge 5$, we are done. If $2p+1=3$, we are back in case 1.

\noindent {\bf Proof of $(ii)$. ~}
We prove the result by contradiction.
Let $G$ be a separable graph of order $q$ and let $\maI_1,...,\maI_q$ be its disjoint maximal independent sets.
Then
for any two nodes $i,j$ of  $G$, the relation $i \pv j$ implies that $i$ and $j$ belong to the same independent set $\maI_k$,
for some $k \in \llbracket 1,q \rrbracket$, but to no other independent set.

First take the contradictory assumption that $G$ induces a pendant
graph $\breve G$, and label the nodes of $\breve G$ as in Figure
\ref{Fig:pendant}. Suppose that $3 \in \maI_i$. All the neighbors of
$3$ in $\breve G$ (and thus in $G$) cannot be in $\maI_i$, so that,
there exist $j,k,\ell \in \llbracket 1,q \rrbracket \setminus
\{i\}$, such that $4\in \maI_j,\,1\in\maI_k$ and $2\in\maI_{\ell}$.
As $4 \pv 1$, we have $j=k$. Similarly, we have $j=\ell$, and thus
$k=\ell$. But, since $1 \v 2$, nodes $1$ and $2$ do not belong to
the same independent set - a contradiction. Therefore, $G$ cannot
induce a pendant graph.

We next assume that $G$ induces the $2p+1$-cycle $\breve G$, with $p \ge 2$ (so that $q \le 2p+1$).
For simplicity, label the nodes of $\breve G$ as 1,2,...,$2p+1$, in a way that
$$1 \,\v\, 2 \,\v\, 3 \,\v\, ... \,\v\, 2p \,\v\, (2p+1)\, \v \,1.$$
For all $j \in \llbracket 1,2p+1 \rrbracket$, let $i_j$ be such that $j \in \maI_{i_j}$. As $1 \v 2,\,2 \v 3$ and $3 \pv 1$, we have $i_1 \ne i_2$ and $i_1 = i_3$.
In general, we see that any odd node $k$ is in $\maI_{i_1}$, whereas any even node $\ell$ is in $\maI_{i_2}$. It follows that nodes $1$ and
$2p+1$ both belong to $\maI_{i_1}$, but since $(2p+1) \,\v\, 1$ we again arrive at a contradiction, implying that $G$ cannot induce an odd cycle of size 5 or more.
\end{proof}

\subsection{Non-chaoticity of Matching Queues} \label{subsec:nonchaotic}
As was already discussed, Theorem \ref{thm:main} is proved by
employing Lemma \ref{Lm:BipSepCycles}, after showing that the
pendant graph and the $5$-cycle graph can be unstable, even if {\sc
Ncond} holds. However, if $\breve G = (\breve \maV, \breve \maE)$ is
induced in $G = (\maV, \maE)$, and if we assume that $i_0$ is an
unstable node of a matching queue on $\breve G$ when it is
considered in isolation, then it might not be unstable in a matching
queue on $G$. (In our case, $\breve G$ is either the pendant graph
of the 5-cycle.) Therefore, to construct an unstable matching queue
on $G$ itself, we must show that the effect of the arrivals to $\maV
\setminus \breve \maV$ can be controlled so that node $i_0$ remains
unstable in the matching queue on $G$. If we think of the effect of
the arrivals of items of $\maV \setminus \breve \maV$, as a
perturbation of the number of items in the matching queue on $\breve
G$ when considered in isolation, then we must show that
perturbations of the matching queue on $\breve G$ do not get
amplified within the matching queue on $G$. We refer to such a
property as {\em non-chaoticity} of matching queues. See Lemma
\ref{LemInclusion2} below for the precise statement.

To prove this non-chaoticity property, we first prove an auxiliary result.
Let $\| . \|$ denote the $1$-norm on $\R^p$, $p \in \ZZ_{++}$,
$$\| x \| =\sum_{i=1}^p |x_i|,\quad x\in \R^p.$$
We say that two matching queues $\left(G,\lambda,\Phi\right)_{\ssc}$ and $\left(G^{\prime},\lambda^{\prime},\Phi^{\prime}\right)_{\ssc}$
such that $|\maV|=|\maV^\prime|$ and $\lambda=\lambda^\prime$
have {\em the same input}, if both are constructed
using the same $|\maV|$ Poisson processes (same sample paths of the arrival process).

For a given matching queue
$\left(G,\lambda,\Phi\right)_{\ssc}$ having a queue process defined on a state space $\mathbb G$, let
$Q^z=\left\{Q^z(t)\,:\,t\ge 0\right\}$ denote the queue process
when the initial condition is $Q^z(0)=z,$ $z \in \mathbb G$.

\begin{lemma}
\label{lemma:Inclusion1} For any matching queue
$\left(G,\lambda,\Phi\right)_{\ssc}$ and any initial conditions
$x,y$ in $\mathbb G$, if the two systems have the same input we have
that
$$\| Q^x(t)-Q^y(t) \|\, \le \, \| x-y \|,\,t\ge 0.$$
\end{lemma}
We remark that the Lipschitz-continuity property stated above
follows (as will become clear in the proof) from a specific property
of the DTMC embedded in arrival-time epoches, known as {\em
non-expensiveness} in the literature on stochastic recursions; see,
e.g., Section 2.11 in \cite{BacBre02}.

\begin{proof}[Proof of Lemma \ref{lemma:Inclusion1}]
Let $T_1 < T_2 < ...$ be the arrival times of elements to the system. With some abuse of notation, denote for the
time being, $Q^x(0)=x,\,Q^y(0)=y$, and for all $n\ge 1$, $Q^x(n):=Q^x(T_n)$ and $Q^y(n):=Q^y(T_n)$.
Since both processes $Q^x$ and $Q^y$ are constant between arrival times, it suffices to show the result at any time $T_n$.
We reason by induction. Let $n\in\ZZ_{+}$, and assume that $$\| Q^x(n)-Q^y(n) \| \le \| x-y \|.$$
Let $j$ be the class of the item drawn at time $T_n$. We have the following alternatives:
\begin{enumerate}
\item Assume that the new item is matched in both systems, with an item of the same class $k \in \maE(j)$.
Then in both cases the $k$th coordinate decreases by one and
we have $\| Q^x(n+1)-Q^y(n+1) \|= \| Q^x(n)-Q^y(n) \|.$

\item Assume that the new item is matched with $k^x$ (respectively, $k^y$) in the system initiated by $x$ (respectively, by $y$),
where $k^x \ne k^y$. Then, by the definition of a priority policy, we must have that
$$\Big((Q^x_{k^x}(n)>0)\mbox{ and }(Q^y_{k^x}(n)=0)\Big) \text{ or }
\Big((Q^x_{k^y}(n)=0)\mbox{ and }(Q^y_{k^y}(n)>0)\Big)$$
(or both), since otherwise, the new arrival of class $j$ would be matched with the same item in both systems.

Assume that we are in the first case, the other one is symmetric.
We have
\bes
\bsplit
& \parallel Q^x(n+1)-Q^y(n+1) \parallel \\
& = \| Q^x(n)-Q^y(n) \| - |Q^x_{k^x}(n)-Q^y_{k^x}(n)| + |(Q^x_{k^x}(n)-1)-Q^y_{k^x}(n)| \\
& \quad - |Q^x_{k^y}(n)-Q^y_{k^y}(n)| + |(Q^y_{k^y}(n)-1)-Q^x_{k^y}(n)| \\
& \le \| Q^x(n)-Q^y(n) \| - Q^x_{k^x}(n) + (Q^x_{k^x}(n)-1) - |Q^x_{k^y}(n)-Q^y_{k^y}(n)| \\
& \quad + |Q^y_{k^y}(n)-Q^x_{k^y}(n)| + 1 \\
& = \,\| Q^x(n)-Q^y(n) \|.
\end{split}
\ees
\item Assume that the new arrival to class $j$ is matched in the system initiated by $x$, say, with $k^x\in \maE(j)$,
but not in the one initiated by $y$ (the other way around is symmetric). Then we must have $Q^y_{k^x}(n)=0$
and in turn,
\bes
\begin{split}
& \| Q^x(n+1)-Q^y(n+1) \| \\
& = \| Q^x(n)-Q^y(n) \| - |Q^x_{k^x}(n)-Q^y_{k^x}(n)| + |(Q^x_{k^x}(n)-1)-Q^y_{k^x}(n)| \\
& \quad - |Q^x_{j}(n)-Q^y_{j}(n)| + |(Q^y_{j}(n)+1)-Q^x_{j}(n)| \\
& \le \| Q^x(n)-Q^y(n) \| - Q^x_{k^x}(n) + (Q^x_{k^x}(n)-1) - |Q^y_{j}(n)-Q^x_{j}(n)| \\
& \quad + |Q^y_{j}(n)-Q^x_{j}(n)| + 1 \\
& = \| Q^x(n)-Q^y(n) \|.
\end{split}
\ees
\item Finally, assume that $Q^x_k(n)=Q^y_k(n)=0$ for any $k\in
\maE(j)$
i.e., the incoming item is not matched upon
arrival in both systems. Then the $j$-th coordinate increases by one
in both cases, so $\| Q^x(n+1)-Q^y(n+1) \|=\| Q^x(n)-Q^y(n) \|.$
\hfill \qedhere
\end{enumerate}
\end{proof}

For a connected graph $G=(\maV,\maE)$ and
$\breve{\maV}\cup\hat{\maV}$ a partition of $\maV$, we denote by
$\breve G$ and $\hat G$ the graphs induced, respectively, by
$\breve{\maV}$ and $\hat{\maV}$ in $G$. Then, the {\em disconnected}
graph $\tilde G$ corresponding to the partition
$\breve{\maV}\cup\hat{\maV}$ is the graph $\tilde G =
\left(\maV,\tilde{\maE}\right)$ such that
$\tilde{\maE}=\maE\setminus \left(\left(\breve{\maV}\times
\hat{\maV}\right) \cup \left(\hat{\maV}\times
\breve{\maV}\right)\right)$. In other words, the graph $\tilde G$ is
obtained from $G$ by erasing the edges between elements of
$\breve{\maV}$ and $\hat{\maV}$; an example is depicted in Figure
\ref{Fig:graphtilde}.

\begin{figure}[h!]
\begin{center}
\begin{tikzpicture}
\draw[-,very thin] (1,1) -- (2,0); \draw[-,very thin] (1,1) --
(3,0); \draw[-,very thin] (1,1) -- (4,0); \draw[-,very thick] (2,2)
-- (2,0); \draw[-,very thick] (2,2) -- (3,0); \draw[-,very thick]
(3,2) -- (3,0); \draw[-,very thin] (4.7,1.5) -- (1,1); \draw[-,very
thin] (4.7,1.5) -- (2,0); \draw[-,very thin] (3,2) -- (4,0);
\fill (1,1) circle (2pt); \fill (2,2) circle (2pt); \fill (3,2)
circle (2pt); \fill (2,0) circle (2pt); \fill (3,0) circle (2pt);
\fill (4,0) circle (2pt); \fill (4.7,1.5) circle (2pt);
%
%
\draw[-,very thin] (7,1) -- (10,0); \draw[-,very thick] (8,2) --
(8,0); \draw[-,very thick] (8,2) -- (9,0); \draw[-,very thick] (9,2)
-- (9,0);
\draw[-,very thin] (10.7,1.5) -- (7,1);
%
\fill (7,1) circle (2pt); \fill (8,2) circle (2pt); \fill (9,2)
circle (2pt); \fill (8,0) circle (2pt); \fill (9,0) circle (2pt);
\fill (10,0) circle (2pt); \fill (10.7,1.5) circle (2pt);
\end{tikzpicture}
\caption[smallcaption]{Right, the disconnected graph $\tilde G$, if
$\breve G$ is the "N" graph.} \label{Fig:graphtilde}
\end{center}
\end{figure}
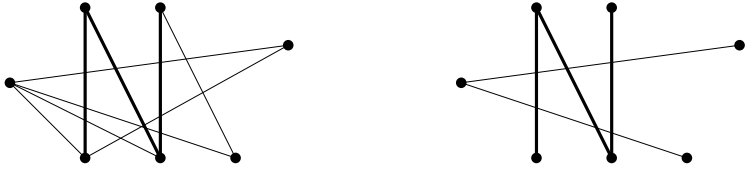

Consider a connected graph $G$, a partition $\breve{\maV}\cup\hat{\maV}$ of $\maV$, and the disconnected graph $\tilde
G$ as defined above. For two priority matching policies $\Phi$ on
$G$ and $\tilde \Phi$ on $\tilde G$, we say that the restrictions to
$\breve G$ of $\Phi$ and $\tilde \Phi$ coincide, if (recall
(\ref{eq:Phiji}))
\[\Phi_j(i)\cap\breve{\maV}=\tilde\Phi_j(i)\cap\breve{\maV},\,i,j \in \breve
{\maV}.\] That is, for any elements $i,j,k$ of $\breve{\maV}$,
$j$ prioritizes $k$ over $i$ according to $\Phi$ if and only if it
does so according to $\tilde\Phi$. We then have the following
result.

\begin{lemma} \label{LemInclusion2}
Let $G=(\maV,\maE)$ be a connected graph,
$\breve{\maV}\cup\hat{\maV}$ be a partition of $\maV$ and $\tilde G$
be the corresponding disconnected graph. Consider the two matching
queues $\Sigma:=\left(G,\lambda,\Phi\right)_{\ssc}$ and
$\tilde\Sigma:=(\tilde G,\lambda,\tilde \Phi)_{\ssc}$, where the
restrictions to $\breve G$ of the priority matching policies $\Phi$
and $\tilde \Phi$ coincide. Let $Q$ and $\tilde Q$ be the respective
queue processes of $\Sigma$ and $\tilde \Sigma$. Then, if
$Q(0)=\tilde Q(0)$, $Q_j(0)=\tilde Q_j(0)=0$ for all $j \in
\hat{\maV}$, and the two systems have the same input, we have that
$$\sum_{i\in \breve{\maV}} \left|Q_i(t) - \tilde Q_i(t)\right| \le \hat N(t),\quad t \ge 0,$$
where $\hat N(t)$ is the number of arrivals of items to $\hat{\maV}$
up to $t$.
\end{lemma}

\begin{proof}
Let $\hat T_n$, $n \ge 1$, be the points of $\hat N$ (i.e., the arrival times of elements of $\hat{\maV}$).
Let also $\hat U_0=0$ and for all $n \ge 0$,
$$\hat U_{n+1}=\inf\left\{t \ge \hat U_n;\,\scriptsize{\mbox{a matching occurs at }t\mbox{ in }\Sigma \mbox{ between an element of }
\breve {\maV} \mbox{ and one of }\hat {\maV}}\right\}.$$ Notice that
the times $\hat U_n$, $n\ge 1$ can either coincide with points of
$\hat N$, or with arrival times of elements of $\breve {\maV}$ that
are matched with an element of $\hat {\maV}$ in $\Sigma$.

Let $n\ge 0$. In the time interval $\left(\hat U_n,\hat
U_{n+1}\right)$, the restrictions to $\breve G$ of $\Sigma$ and of
$\tilde \Sigma$ both behave exactly as the matching queue
$\left(\breve G,\lambda_{\breve{\maV}},\breve \Phi\right)_{\ssc}$,
where $\breve \Phi$ is the restriction of $\Phi$ and $\tilde \Phi$
to $\breve G$. Therefore we can apply Lemma \ref{lemma:Inclusion1}
to the latter model and to the initial conditions
$Q_{\breve{\maV}}\left(\hat U_{n}\right)$ and $\tilde
Q_{\breve{\maV}}\left(\hat U_{n}\right)$, where we recall that
$X\left(t^-\right)$ denotes the left-limit of $X$ at $t$. We
conclude that for any $n \ge 0$,
\begin{equation}
\label{eq:inclusion1} \sum_{i\in \breve{\maV}} \left|Q_i\left(\hat
U_{n+1}^-\right) - \tilde Q_i\left(\hat U_{n+1}^-\right)\right| \le
\sum_{i\in \breve{\maV}} \left| Q_i\left(\hat U_{n}\right) - \tilde
Q_i\left(\hat U_{n}\right)\right|.
\end{equation}
Furthermore, we have the following alternatives:
\begin{enumerate}
\item If $\hat U_{n+1}$ is an arrival time of an item of class in $\hat {\maV}$ (it coincides with some $\hat T_k$),
this item is matched immediately with an items of a class in $\breve{\maV}$
(say, of class $j$), which leaves the buffer of $\Sigma$. Hence, in the
restriction of $Q$ to $\breve{\maV}$, all coordinates remain
unchanged except for the $j$th coordinate, which decreases by 1. On the
other hand, the restriction of $\tilde Q$ to $\breve {\maV}$ does not
change.
\item If $\hat U_{n+1}$ is an arrival time of an item of class $j\in\breve {\maV}$ that is matched immediately with an element of
$\hat {\maV}$ in $\Sigma$, then the buffer content of $\Sigma$
restricted to $\breve {\maV}$ does not change, but that of $\tilde
\Sigma$ does:
    \begin{itemize}
    \item[2a.] if the arriving item is matched in $\tilde \Sigma$ with an item of class $k \in
    \maE(j)\cap \breve {\maV}$, the $k$th coordinate of $\tilde Q$ decreases by 1, while all other coordinates remain unchanged;
    \item[2b.] if the arriving item does not find a match in $\breve {\maV}$ in $\tilde
    \Sigma$, it is stored in the buffer and the $j$th coordinate of $\tilde Q$ increases by 1, while all other coordinates remain unchanged.
    \end{itemize}
\end{enumerate}
In all cases, we obtain that
\begin{equation}
\label{eq:inclusion2} \sum_{i\in \breve{\maV}} \left| Q_i\left(\hat
U_{n+1}\right) - \tilde Q_i\left(\hat U_{n+1}\right)\right| \le
\sum_{i\in \breve{\maV}} \left| Q_i\left(\hat U_{n+1}^-\right) -
\tilde Q_i\left(\hat U_{n+1}^-\right)\right| +1.
\end{equation}
Finally, gathering (\ref{eq:inclusion1}) and (\ref{eq:inclusion2}), as $Q(0)$ and $\tilde Q(0)$ coincide, we obtain that for any $t\ge 0$,
\begin{equation} \label{eq:inclusion}
\sum_{i\in \breve{\maV}} | Q_i\left(t\right) - \tilde Q_i\left(t\right)| \le \sum_{n\ge 1} \ind_{\{\hat U_n \le t\}}.
\end{equation}
Finally, observe that the instants $\hat U_n$ are departure
times of items of class in $\hat{\maV}$. As there are no such items
in storage initially, the number of such instants up to $t$ cannot
exceed the number of arrivals of items of class in $\hat{\maV}$ up
to $t$. Thus we have
\begin{align*}
\sum_{n\ge 1} \ind_{\{\hat U_n \le t\}}
& \le \sum_{n\ge 1} \ind_{\{\hat T_n \le t\}},
\end{align*}
which, together with (\ref{eq:inclusion}), concludes the proof.
\end{proof}

\subsection{Proof of Theorem \ref{thm:main}} \label{subsecProofMain}
We are now in position to prove Theorem \ref{thm:main}. The strategy
is the following: We fix a non-bipartite and non-separable graph $G
\in \mathscr G_7^c$. Lemma \ref{Lm:BipSepCycles} entails that such a
graph induces particular types of graphs, for which the
corresponding matching queues $\left(G,\lambda,\Phi\right)_{\ssc}$
are proved to be possibly unstable for some $\lambda \in \mbox{{\sc
Ncond}}_{\ssc}(G)$ and some matching policy $\Phi$ (Sections
\ref{subsec:pendant} and \ref{subsec:5cycle}).
The instability of the matching queue under consideration can then be deduced 
from the non-chaoticity property in Lemma \ref{LemInclusion2}.

\begin{proof}[Proof of Theorem \ref{thm:main}]
Let $G=(\maV,\maE)$ be a non-bipartite and non-separable graph in
$\mathscr G_7^c$. By Lemma \ref{Lm:BipSepCycles}, $G$ induces a
pendant graph or an odd cycle of size $5$. Let $\breve
G=(\breve{\maV},\breve{\maE})$ be that induced subgraph. Then there
exists an arrival-rate vector $\breve{\lambda} \in
\left(\R_{++}\right)^{|\breve{\maV}|}$ and a matching policy $\breve
\Phi$, such that the matching queue $\left(\breve G,\breve
\lambda,\breve \Phi\right)_{\ssc}$ is unstable, whereas $\breve
\lambda \in \mbox{{\sc Ncond}}_{\ssc}(\breve G)$. (This latter claim
follows from Proposition \ref{prop:nonsufficientpendant} if $\breve
G$ is a pendant graph, and from Proposition
\ref{prop:nonsufficient5cycle} if $\breve G$ is a 5-cycle.) We fix
the latter $\breve \lambda$ until the end of the proof and set
(recall (\ref{eq:deflambda}))
\begin{equation}
\tau := \min \Bigl\{\overline{\breve\lm}_{\breve \maE\left(\breve \maI\right)} - \overline{\breve\lm}_{\breve \maI} :\, \breve \maI\in \I(\breve G)\Bigl\}.\label{eq:defalpha}
\end{equation}
Let $\hat{\maV}= \maV \setminus \breve{\maV}$ and denote $\hat
G=\left(\hat{\maV},\hat{\maE}\right)$ the induced subgraph in $G$.
In view of Proposition \ref{pro:Pendant} and Proposition \ref{pro:5cycle}, 
there exists a node $i_0\in \breve V$ ($i_0=4$ if $\breve G$ is a
pendant graph and $i_0=5$ if $\breve G$ is a $5$-cycle) and a
measure $\breve\pi$ on $\mathbb E_{2}$ such that the drift of the
$i_0$-coordinate of the fluid limit reads
\begin{equation}
\label{eq:defbeta} \beta:=\breve\lm_{i_0} - \sum_{j\in
\maE\left(i_0\right)} \breve \lm_{j}\breve\pi \Bigl(\maP^{\maS}_{j}(i_0)\Bigl)>0.
\end{equation}
Set $\gamma={1\over 2}\left(\tau \wedge \beta\right)$, and let $\lambda \in \R_{++}^{|\maV|}$ satisfy
\begin{equation}
\label{eq:hypoepsilon2}
\left\{\begin{array}{ll}
\lm_{\breve {\maV}} &= \breve \lm;\\
\bar{\lm}_{\hat {\maV}} &\le \gamma.
\end{array}\right.
\end{equation}
We first prove that $\lambda \in \mbox{{\sc Ncond}}_{\ssc}(G).$
For $\maI \in \I(G)$, observe that
$$\breve \maE\left(\maI \cap \breve {\maV}\right) \cup \hat \maE\left(\maI \cap \hat {\maV}\right) \subset \maE(\maI).$$
Therefore,
\begin{equation}
\label{eq:Niki1}
\bar{\lm}_{\maE(\maI)} - \bar{\lm}_{\maI}
\ge \bar{\lm}_{\breve \maE(\maI \cap \breve {\maV})} + \bar{\lm}_{\hat \maE(\maI \cap \hat {\maV})} - \bar{\lm}_{\maI \cap \breve {\maV}} - \bar{\lm}_{\maI \cap \hat {\maV}}.
\end{equation}
First, as $\breve G$ is induced in $G$, $\maI \cap \breve {\maV} \in \I(\breve G)$. Thus, (\ref{eq:defalpha}) implies that
$$\bar{\lm}_{\breve \maE(\maI \cap \breve {\maV})}-\bar{\lm}_{\maI \cap \breve {\maV}} > \gamma.$$
Moreover, from (\ref{eq:hypoepsilon2}) we clearly have that
$$\bar{\lm}_{\hat \maE(\maI \cap \hat {\maV})} - \bar{\lm}_{\maI \cap \hat {\maV}} \ge -\gamma.$$
These two observations, together with (\ref{eq:Niki1}), yield
$\bar{\lm}_{\maE(\maI)} - \bar{\lm}_{\maI} >0$. Therefore $\lm \in \mbox{{\sc Ncond}}_{\ssc}(G).$


It remains to construct a matching policy $\Phi$ on $G$ rendering
the matching queue $\left(G,\lambda,\Phi\right)_{\ssc}$ unstable. To
this end, it suffices to consider any $\Phi$ whose restriction to
$\breve G$ is $\breve \Phi$. Consider the process $\tilde Q$
constructed in Lemma \ref{LemInclusion2}, associated with the
disconnected graph corresponding to the partition
$\breve{\maV}\cup\hat{\maV}$. If $Q(0)$ and $\tilde Q(0)$ are equal,
and satisfy $Q_j(0)=\tilde Q_j(0)=0$ for all $j \in \hat{\maV}$,
then it follows from Lemma \ref{LemInclusion2} that
\begin{equation}
\label{eq:final1}
Q_{i_0} \ge_{st} \tilde Q_{i_0}-\hat N,
\end{equation}
where $\hat N$ denotes again the arrival process of items of class
in $\hat {\maV}$. For $n\ge 1$, let $\hat N^n$ be a Poisson process
with intensity $n\bar\lm_{\hat {\maV}}$. Then, for any initial
condition $\breve Q^n(0)$ such that $\breve Q_{i_0}^n(0)=nx$,
the following convergence holds in $\mathbb D^{\mid \breve
{\maV}\mid}$,
\begin{multline*}
\proce{{1\over n}\breve Q^n(t) - {1\over n} \hat N^n(t) \gre_{i_0}}\\
\Rightarrow \proce{\left(x+\breve\lm_{i_0} - \sum_{j\in
\maE\left(i_0\right)} \breve \lm_j\breve\pi
\left(\maP^{\maS}_{j}(i_0)\right)-\bar\lm_{\hat{\maV}}\right)t\gre_{i_0}}.
\end{multline*}
Together with (\ref{eq:defbeta}) and (\ref{eq:hypoepsilon2}), the
above convergence implies that the Markov process $\breve Q - \hat N
\mathbf e_{i_0} $ is transient and that $\breve Q_{i_0} -
\hat N \Ra \infty$ as $t \to \infty$.

Finally, observe that by definition, the restriction to
$\breve{\maV}$ of $\tilde Q$ has the same distribution as $\breve Q$
if $\breve Q(0)$ is set equal to the restriction to
$\breve{\maV}$ of $\tilde Q(0)$. Thus, the $\ZZ_+^{\mid
\maV\mid}$-valued Markov process $\tilde Q - \hat N\gre_{i_0}$ is
transient and its $i_0$-coordinate converges in distribution to
$\infty$ as $t \tinf$. By (\ref{eq:final1}), this is also the case
for the $i_0$-coordinate of $Q$, so that $Q$ is transient.
\end{proof}

\begin{remark}{\em
Observe that, for any non-bipartite and non-separable graph $G$, the
the proof of Theorem \ref{thm:main} not only shows the existence of
a non-maximal priority policy on $G$, but also provides a simple way
of constructing that policy. Specifically, we have proven that for
{\em any} priority matching policy $\Phi$, if the restriction of
$\Phi$ to the induced sub-graph $\breve G$ is non-maximal for
$\breve G$, then $\Phi$ is also non-maximal for $G$. Consequently,
(i) if $G$ induces a pendant graph (whose nodes are labeled
as in Figure \ref{Fig:pendant}), then any priority policy $\Phi$ on
$G$ such that node 3 prioritizes nodes 1 and 2 over node 4 is
non-maximal; (ii) if $G$ induces a 5-cycle (whose nodes are labeled
as in Figure \ref{Fig:5cycle}), then any priority policy $\Phi$ on
$G$ such that node 3 prioritizes node 1 over 5, node 4 prioritizes
node 2 over 5, node 1 prioritizes 2 over 3 and node 2 prioritizes 1
over 4, is non-maximal.
}\end{remark}

\section{Proof of the FWLLN} \label{secFluidProof}

Our proof of the FWLLN will follow the pre-compactness approach \cite{Billingsley,WhittBook}.
In particular, we will show that $\{\barq^n : n \ge 1\}$ is tight in $\D^{|\maV|}$ and uniquely characterize the limit.


Recall that $\mGS$ defined by (\ref{eq:defmGS}) denotes the state space of $S^n$ for all $n \ge 1$.
For $\delta > 0$, let $\maM := \maM(\maL_\delta)$ denote the space of (finite) measures on the space $\maL_\delta := [0,\delta) \times \mGS$
such that $\mu([0,t] \times \mGS) = t$ for all $\mu \in \maM$
and $t \in [0, \delta)$, endowed with the Prohorov metric \cite{Billingsley}.
Next, define a sequence of random elements
$\{\nu^n : n \ge 1\} \subset \maM$ via
\bequ \label{randomMeasures}
\nu^n([0,t] \times \mathcal Z) := \int_0^t \ind_{\mathcal Z}(S^n(u-)) du, \quad 0 \le t < \delta , \quad \mathcal Z \subseteq \mGS.
\eeq


\begin{lemma} \label{LmTight}
If $\barq^n(0)$ is tight in $\RR^{|\maV|}$, then $\{(\barq^n, \nu^n) : n \ge 1\}$ is tight in
$\D^{|\maV|}[0, \delta) \times \maM$, for $\delta$ in Lemma \ref{lmBndFluid}.
Moreover, for any limit point $(\barq, \nu)$, 
\bequ \label{limRandomMeasure}
\nu\left([0,t] \times \mathcal Z\right) = \int_0^t p_u(\mathcal Z) \,du; \quad 0 \le t < \delta, ~\mathcal Z \subseteq \mGS
\eeq
for some family of probability measures $\{p_u : 0 \le u \le t < \delta\}$ on $\mGS$.
\end{lemma}

\begin{proof}
Tightness of the vector $\{(\barq^n, \nu^n) : n \ge 1\}$ follows from the tightness of each component separately; e.g., Theorem 11.6.7 in \cite{WhittBook}.
We start by showing that $\{\nu^n : n \ge 1\}$ is tight in $\maM$. 

It follows from \eqref{XiEqDistSn} that, for any compact set $\mathcal Z \subset \mGS$ and $0 \le t < \delta$,
\bequ \label{TightCondRM}
\bsplit
E[\nu^n([0,t] \times \mathcal Z)] & = E\left[ \int_0^t \ind_{\mathcal Z}(S^n(u-)) du \right] \\
& = E\left[ \int_0^t \ind_{\mathcal Z}(\chi^n(u-)) du \right] \\
& = E\left[ \frac{1}{n}\int_0^{nt} \ind_{\mathcal Z} (\chi(u-)) du \right] \\
& \ra P(\chi(\infty) \in \mathcal Z) t \\
& = \pi(\mathcal Z) t.
\end{split}
\eeq
The limit in \eqref{TightCondRM} holds as $n\tinf$ and follows from the ergodicity of $\chi$ and the bounded convergence theorem.

Take $\ep > 0$ and $t \in [0, \delta)$. Observe that, for any finite $N$ and all $n \le N$, the right-hand side of the second equality in \eqref{TightCondRM}
implies that for large-enough $c \in \ZZ_+$ and for $K_c := [0,c]^{|\maS|}$, it holds that $E[\nu^n([0,t] \times K_c)] \ge (1-\ep) t$.
The limit in \eqref{TightCondRM} shows that this latter inequality holds for all $n$, i.e., we have
\bes
\inf_n E[\nu^n([0,t] \times K_c)] \ge (1-\ep) t,
\ees
for a sufficiently large $c$.
Hence, tightness of $\{\nu^n : n \ge 1\}$ follows from Lemma 1.3 in \cite{KurtzAP}.
The fact that any limit point of this tight sequence has the form in \eqref{limRandomMeasure} follows from Lemma 1.4 in \cite{KurtzAP}.

We next show that $\{\barq^n : n \ge 1\}$ is $\mathbb C$-tight in $\D^{|\maV|}$,
i.e., it is tight and any of its limit points is continuous.
We first note that, since all the jumps in $\barq^n$ are of size $1/n$, and are thus converging to $0$ as $n \tinf$,
we can work with the modulus of continuity defined for continuous functions (see, e.g., \cite[p.\ 123]{Billingsley})
\bes
w(y, \eta, T) := \sup\biggl\{|y(t_2) - y(t_1)| : 0 \le t_1 \le t_2 \le T, ~ |t_2 - t_1| \le \eta\biggl\},~\eta > 0,~ T > 0.
\ees
Hence, we can establish the result by applying Theorem 11.6.3 in p. 389 in \cite{WhittBook}.
Conditions (6.3) in that theorem, namely, tightness of $\left\{\barq^n(0) : n \ge 1\right\}$ in $\RR^{|\maV|}$ is assumed.
To show that condition (6.4) in \cite[Theorem 11.6.3]{WhittBook} also holds, observe that, for all $i \in \maV$ and $n \ge 1$,
\bes
|\barq^n_i(t_2) - \barq^n_i(t_1)| \le \frac{N^n_i(t_2) - N^n_i(t_1)}{n} + \frac{Z^n_i(t_2) - Z^n_i(t_1)}{n},
\ees
where $N^n_i$ denotes the time-scaled Poisson arrival process to node $i$ (see \S \ref{secFWLLN}), and
$Z^n_i(t)$, $n \ge 1$, is the number of type-$i$ items that were matched (and left the system)
by time $t$. Then, we have that
\bes
w(\barq^n_i, \eta, T) \le w\left(N^n_i/n, \eta, T\right) + w\left(Z^n_i/n, \eta, T\right), \quad \eta, T > 0.
\ees
It follows from the representation of $Q^n_i$ in \eqref{repForQn} that
the oscillations of  $\barq^n_i$ are bounded by those of the time-scaled Poisson processes $N^n_i$ and $N^n_j$, $j \in \maE(i)$. Hence,
\bequ \label{modulusBd1}
w(\barq^n_i, \eta, T) \le w(N^n_i/n, \eta, T) + \sum_{j \in \maE(i)}w(N^n_j/n, \eta, T), \quad j \in \maE(i).
\eeq
By the FWLLN for Poisson processes, all the moduli of continuity for the scaled Poisson processes in \eqref{modulusBd1} are controlled, so that,
for every $\ep > 0$ and $\zeta > 0$, there exists $\eta > 0$ and $n_0 \in \ZZ_+$, such that
\bes
P(w(\barq^n_i, \eta, T) \ge \ep) < \zeta \qforallq n \ge n_0.
\ees
Thus, $\left\{\barq^n_i : n\ge 1\right\}$ is $\mathbb C$-tight for each $i \in \maV$, implying that $\left\{\barq^n : n\ge 1\right\}$ is $\mathbb C$-tight in $\D^{|\maV|}$,
and in particular, it is $\mathbb C$-tight in $\D^{|\maV|}[0,\delta)$.
\end{proof}



In the proof of Theorem \ref{thFWLLN} we employ the standard result that,
for all $i \in \maV$ and $n \ge 1$, the following process is a square integrable martingale (with respect to the filtration generated by the Poisson processes);
see, e.g., \cite{PTW07}.
\bequ \label{repMartQ}
\bsplit
M^n_i(t) & := \sum_{j\in \maE(i)}\int_0^t \ind_{\maN_i}(Q^n(s-)) \ind_{\maP_{j}(i)}(Q^n(s-))\, dN^n_j(s) \\
& \quad - \int_0^t \ind_{\maN_i}(Q^n(s-)) \left(\sum_{j\in \maE(i)} n\lm_j \ind_{\maP_{j}(i)}(Q^n(s-))\right)\, ds, \quad t \ge 0.\\
\end{split}
\eeq

\begin{proof}[Proof of Theorem \ref{thFWLLN}]
The assumed convergence of the initial condition $\barq^n(0)$ implies that it is also tight in $\RR^{|\maV|}$.
To characterize the limit, let us first consider $\barq^n_{i_0}$.
Fix $t < \delta$ and $n \ge n_0$, where $n_0$ is defined by Lemma \ref{lmBndFluid}.
Then \eqref{repForQn} can be written as follows for node $i_0$,
\bequ \label{eq:martn0}
\bsplit
\barq^n_{i_0}(t) & = \barq^n_{i_0}(0) + \int_0^{t} \ind_{\maO_{i_0}}(Q(s-)) dN^n_{i_0}(s) \\
& \quad - \sum_{j\in \maE\left(i_0\right)} \lm_j\int_0^t \ind_{\maP_j(i_0)}\left(Q^n(s)\right)\,ds -  M^n_{i_0}(t)/n,
\end{split}
\eeq
for $M^n_{i_0}$ in \eqref{repMartQ}.
Now, observe that for all $j \in \maE(i_0)$ and $s\le t$,
we have that $Q^n_k(s)=0$ for all $k \in \maE(i_0)\cap\maE(j)$.
Therefore, for all such $n$, $j$ and $s$ we have that
\[\ind_{\maP_j(i_0)}\left(Q^n(s)\right) = \prod_{\substack{\ell \in \llbracket 1,|\maS| \rrbracket:\\i_\ell \in \Phi_j(i_0)}}\ind_{\{0\}}\left(Q_{i_\ell}(s)\right)
=\ind_{\maP^{\maS}_j(i_0)}\left(S^n(s)\right).\] Thus, we obtain from
(\ref{eq:martn0}) that, for all $n\ge n_0$ and $t <\delta$,
\bequ \label{eq:martn} \bsplit
\barq^n_{i_0}(t) & = \barq^n_{i_0}(0) + \int_0^{t} \ind_{\maO_{i_0}}(Q(s-)) dN^n_{i_0}(s) \\
& \quad - \sum_{j\in \maE\left(i_0\right)} \lm_j\nu^n \left([0,t]\,
\times \maP^{\maS}_{j}(i_0)\right) - M^n_{i_0}(t)/n,
\end{split}
\eeq for $\nu^n$ in \eqref{randomMeasures}. Now, the equality in
distribution \eqref{XiEqDistSn} and the ergodicity of the CTMC
$\chi$ imply (similarly to \eqref{TightCondRM}) that \bequ
\label{nuLim} \nu^n([0,T] \times \mathcal{Z}) \Ra \pi(\mathcal Z)T
\qasq n \tinf, \quad \mathcal Z \subseteq \mGS, \eeq for $\pi$ in
\eqref{piDist}. By Lemma \ref{lmBndFluid} and the FWLLN for the
Poisson process, the second argument to the right of the equality in
\eqref{eq:martn} converges weakly to $\lm_{i_0} e$ in $\D[0,
\delta)$, where $e$ denotes the identity function $e(t) = t$. Since
$\ind_{\maN_i}(Q(s-))$ is identically equal to $1$ for all $n$ large
enough, again by Lemma \ref{lmBndFluid}, and
$M^n_{i_0}/n \Ra 0e$ in $\D[0, \delta)$ as $n\tinf$, 
by virtue of Doob's martingale inequality, the limit \eqref{FluidLim} follows from
\eqref{nuLim} and Lemma 7.3 in \cite{PWfluid} (a simple extension to the continuous mapping theorem).

Next, recall (\ref{eq:scalingS}) and (\ref{eq:scalingChi}),
and consider $\{\bar S^n : n \ge 1\}$. By Assumption
\ref{AssErgo}, each element $\chi^n, n \ge 1$, is an ergodic CTMC,
and thus $\{\bar \chi^n : n \ge 1\}$ is a $\mathbb C$-tight sequence
of ergodic CTMC's.
It follows
from Proposition 9.9 in \cite{RobertBook} that there exists an
a.s.-finite time $\maT$, such that $\bar{\chi}^n \Ra 0e$ in
$\D^{|\maS|}(\maT,\infty)$ as $n\tinf$. In particular, with $d_P$
denoting the Prohorov metric \cite{Billingsley, Durrett} (here,
denoted in terms of the distance between the random elements instead
of the distance between their corresponding probability measures in
the underlying probability space), it holds that, for any $t > \maT$
and for any $\ep > 0$,
\bequ \label{d_P} d_P(\bar
\chi^n(t), 0) < \ep/2, \quad \text{for all $n$ large enough.}
\eeq
 Since any limit point of the
tight sequence $\{\bar \chi^n : n \ge 1\}$ is continuous, if it is
ever larger than $0$, then it must be strictly positive over an
interval.

Fix $\ep > 0$ and let $\|\cdot\|_{tv}$ denote the total-variation norm; see, e.g., \cite{Durrett}.
It follows from \eqref{scalingChi} that for any $t > 0$, there exists
$n_1$, such that $\|\chi^n(t) - \chi(\infty)\|_{tv} < \ep/2$ for
all $n > n_1$. Hence, by the triangle inequality, for any $s < \maT$
and $t > \maT$ there exists $n_2$, such that for all $n > n_2$,
\bequ \label{d_tv} \|\chi^n(s) - \chi^n(t)\|_{tv} < \ep/2, \text{
or, equivalently, ~} \|\bar\chi^n(s) - \bar\chi^n(t)\|_{tv} < \ep/2.
\eeq Now, since the Prohorov metric and the total-variation metric
(induced by the total variation norm) are equivalent in discrete
state spaces, \eqref{d_tv} implies that $d_P(\bar \chi^n(s), \bar
\chi^n(t)) < \ep/2$ for all $n > n_2$. Together with \eqref{d_P} and
the triangle inequality, we obtain \bes d_P(\bar \chi^n(s), 0) <
\ep, \quad 0 < s < \maT. \ees The pointwise convergence of $\bar
\chi^n$ to $0$ implies that no limit point of the $\mathbb C$-tight
sequence $\{\bar \chi^n : n \ge 1\}$ can be strictly positive over
an interval, so that $\bar \chi^n \Ra 0e$, and in turn, by
\eqref{XiEqDistSn}, $\bar S^n \Ra 0e$ as $n \tinf$ in
$\D^{|\maS|}[0, \delta)$.

\paragraph{Increasing the interval of convergence}
The representation of $\barq^n_{i_0}$ in \eqref{eq:martn0} holds as long as $\barq^n_{i_0} > 0$, and in particular, over $[0,\rho^n)$.
Since $P(\barq^n_{i_0}(\delta-)>0) \ra 1$ as $n\tinf$ we conclude from the $\mathbb C$-tightness of $\barq^n$ over $[0,\delta]$
that the convergence of $\barq^n$ in fact holds over $[0, \delta]$.
We can then treat $\barq(\delta)$ as an initial condition, and apply Lemma \ref{lmBndFluid} for this new initial
condition to conclude that there exists a $\delta_2 > \delta$ such that $\barq^n_{i_0} > 0$ w.p.1 over $[0, \delta_2)$
for all $n$ large enough. Hence, the FWLLN holds over $[0, \delta_2)$ as well.
Repeating the same arguments inductively, we can continue increasing the interval of convergence as long as $\barq_{i_0}$ is guaranteed to be strictly positive,
where, in the induction step $k$ we take $\barq^n_{i_0}(\delta_k)$ as an initial condition and apply Lemma \ref{lmBndFluid} to find a $\delta_{k+1} > \delta_k$
such that $\barq^n \Ra \barq$ in $\mathbb D^{|\maV|}[0,\delta_{k+1})$.
If \eqref{CondDecrease} does not hold, then it follows from \eqref{FluidLim} that $\bar Q_{i_0}$ is nondecreasing.
Necessarily, $\rho^n \Ra \infty$ as $n\tinf$, and the convergence of $\barq^n$
to $\barq$ can be extended indefinitely. On the other hand,
if \eqref{CondDecrease} does hold, then the fluid limit $\barq_{i_0}$ is strictly decreasing.
The first passage time $\rho^n$ in \eqref{rho^n} is a continuous mapping by, e.g., Theorem 13.6.4 in \cite{WhittBook},
so that $\rho^n \Ra \rho$ in $\RR$ as $n\tinf$, for $\rho$ in \eqref{rho}, and the convergence of $\bar Q^n$ can be extended from $[0,\delta)$
to $[0, \rho)$.
\end{proof}

\section{Uniform matching policy} \label{sec:U}
Our main result shows that, for $G$ in $\mathscr G_7^c$, there
always exists a ``bad'' choice of priority matching policy, leading
to a stability region that is strictly smaller than
$\textsc{Ncond}_{\ssc}(G)$. In this section we show that the methods developed to prove this result can be
applied to other policies. In particular, we now consider a matching
policy in which the choice of which class to match with an arriving class-$i$ is
drawn uniformly at random from the {\em available} classes in $\maE(i)$ at the arrival epoch.
We refer to this policy as {\em uniform}, and denote it by $\Phi=\sU$.
To formally describe $\sU$, let $t$ be an arrival epoch of a class-$i$ item, and consider the set
\bequ \label{setU}
\maU_i(t) = \left\{j \in \mathcal E(i)\,:\,Q_j(t) >0\right\}.
\eeq
\begin{itemize}
\item[(i)] If $\maU_i(t)=\emptyset$, then no matching occurs, and the arriving item is placed in the buffer.
\item[(ii)] If $\maU_i(t)\ne\emptyset$, then the matching class is chosen uniformly at random, namely, the arriving class-$i$ item will be matched with
a class-$j$ item with probability $1/|\maU_i(t)|$, for each $j \in \maU_i(t)$.
\end{itemize}
The main question Theorem \ref{thm:main} answered was whether the
choice of the (admissible) matching policy affects the stability
region of matching queues having non-bipartite and non-separable
graphs.
The next result demonstrates that 
non-maximality of such graphs is not restricted to strict priority policies.
\begin{proposition} \label{propGR}
The only graphs in $\mathscr{G}_7^c$ for which $\textsc{Ncond}_{\ssc}(G)$ is non-empty and the policy $\sU$ is maximal are separable of order $3$ or more.
\end{proposition}


The proof of Proposition \ref{propGR} follows the same steps of the proof of
Theorem \ref{thm:main}. We therefore specify only the arguments in the proof that need to be modified.
The main step that needs to be modified is the proof of Lemma \ref{lemma:Inclusion1}, which needs to be adapted to the policy $\sU$.
To this end, as in the proof of Lemma \ref{lemma:Inclusion1}, we must couple two systems with
initial buffer contents $x$ and $y \in \mathbb G$, such that
both systems are fed by the same Poisson processes. (We henceforth refer to those systems as ``system $x$'' and ``system $y$''.)

Let $\left\{T_n:\,n \ge 0\right\}$ and $\left\{C_n:\,n \ge 0\right\}$ denote, respectively,
the sequences of arrival times and of the classes of the entering items, in arrival order.
Consequently, for any $n\ge 0$, the $n$th arriving item makes a uniform choice from the
set $\maU^x_{C_n}\left(T_n\right)$ in system $x$ and the set $\maU^y_{C_n}\left(T_n\right)$ in system $y$,
where $\maU^x_{C_n}\left(T_n\right)$ and $\maU^y_{C_n}\left(T_n\right)$
denote the sets defined in \eqref{setU} for systems $x$ and $y$, respectively.
In the present case, the difficulty stems
from the fact that the sets $\maU^x_{C_n}\left(T_n\right)$ and $\maU^y_{C_n}\left(T_n\right)$ {\em a priori} differ,
even though both systems are constructed with the same input $\left\{(T_n,C_n):\,n \ge 0\right\}$.
Nevertheless, we can couple these two systems as follows.

Let
$\left\{K^x_n,\,n\ge 0\right\}$ and $\left\{K^y_n,\,n\ge 0\right\}$ denote two independent sequences of independent random variables,
where, for all $n \ge 0$, $K^x_n$ and $K^y_n$ follow the discrete uniform distribution on $\maU^x_{C_n}\left(T_n\right)$ and $\maU_{C_n}^y\left(T_n\right)$, respectively.
Set $K^x_n=0$ (resp. $K^y_n=0$) if $\maU^x_{C_n}\left(T_n\right)=\emptyset$ (resp., $\maU^y_{C_n}\left(T_n\right)=\emptyset$), and
for all $n\ge 0$, denote the event
\[\mathscr U_n=\left\{\left(K^x_n,K^y_n\right) \in \left(\maU_{C_n}^x\left(T_n\right) \cap \maU_{C_n}^y\left(T_n\right)\right)^2 \right\}.\]
Finally, set
\begin{equation}
\label{eq:deftildeU}
\tilde K^y_n = K^x_n \ind_{\mathscr U_n} + K^y_n \ind_{\mathscr U_n^c}.
\end{equation}
Then the random variables $\tilde K^y_n,\,n\ge 0$, are independent, and for all $n\ge 0$, $\tilde K^y_n$ is uniformly distributed on
$\maU_{C_n}^y\left(T_n\right)$.
To see this, observe that, if $\maU_{C_n}^y\left(T_n\right) \ne \emptyset$, then 
\begin{itemize}
\item for all $k \in \maU_{C_n}^y\left(T_n\right)\setminus \maU_{C_n}^x\left(T_n\right)$,
\[P\left(\tilde K^y_n = k\right)=P\left(\{K^y_n = k\} \cap \mathscr U_n^c\right)=P\left(K^y_n = k\right)={1 \over \left|\maU_{C_n}^y\left(T_n\right)\right|};\]
\item for all $k \in \maU_{C_n}^y\left(T_n\right)\cap \maU_{C_n}^x\left(T_n\right)$,
\begin{align*}
&P\left(\tilde K^y_n = k\right)
=P\left(\{\tilde K^y_n = k\}\cap \mathscr U_n\right)+P\left(\{\tilde K^y_n = k\}\cap \mathscr U_n^c\right)\\
&=P\left(\{K^x_n = k\}\cap \{K^y_n \in \maU_{C_n}^x\left(T_n\right) \cap \maU_{C_n}^y\left(T_n\right)\}\right)\\
&\quad\quad\quad+P\left(\{K^y_n = k\}\cap \{K^x_n \in \maU_{C_n}^x\left(T_n\right) \setminus \maU_{C_n}^y\left(T_n\right)\}\right)\\
&={1 \over \left|\maU_{C_n}^x\left(T_n\right)\right|}{\left|\maU_{C_n}^x\left(T_n\right) \cap \maU_{C_n}^y\left(T_n\right)\right| \over \left|\maU_{C_n}^y\left(T_n\right)\right|}\\
&\quad\quad\quad+ {1 \over \left|\maU_{C_n}^y\left(T_n\right)\right|}\left(1 - {\left|\maU_{C_n}^x\left(T_n\right) \cap \maU_{C_n}^y\left(T_n\right)\right| \over \left|\maU_{C_n}^x\left(T_n\right)\right|}\right)\\
&={1 \over \left|\maU_{C_n}^y\left(T_n\right)\right|}.
\end{align*}
\end{itemize}
We have the following analogue to Lemma \ref{lemma:Inclusion1},
\begin{lemma}
\label{lemma:Inclusion1GR}
Fix $G$ and the matching policy $\Phi=\sU$.
Let $x$ and $y$ be two elements in the state space $\mathbb G$ of $Q$, and denote by $Q^x$ and $Q^y$
the buffer content processes of the two models having initial values $x$ and $y$, respectively, and respectively fed by the
inputs $\left\{(T_n,C_n,K^x_n)\,:\,n\ge 0\right\}$ and $\left\{(T_n,C_n, \tilde K^y_n)\,:\,n\ge 0\right\}$.
Then for all $t\ge 0$,
$$\| Q^x(t)-Q^y(t) \|\, \le \, \| x-y \|.$$
\end{lemma}
\begin{proof}
We reason by induction, as in the proof of Lemma \ref{lemma:Inclusion1}, keeping the notation therein.
Suppose that we have at time $T_n$, $\| Q^x(n)-Q^y(n) \| \le \| x-y \|$. Then, we are in the following alternative:
\begin{enumerate}
\item On $\mathscr U_n$, we have by construction $K^x_n=\tilde K^y_n$, so the newly arrived item of class $C_n$ is matched with an item
of the same class $K^x_n$ in both systems. We are in the case 1 of the proof of Lemma \ref{lemma:Inclusion1}.
\item On $\mathscr U_n^c$, we have three possible cases:
      \begin{itemize}
      \item If both $\maU_{C_n}^x\left(T_n\right)$ and $\maU_{C_n}^y\left(T_n\right)$ are non-empty, then $K^x_n=k^x$ for some
      $k^x \in \maU_{C_n}^x\left(T_n\right)$ and $K^y_n=k^y$ for some $k^y \in \maU_{C_n}^y\left(T_n\right)$.
      However, it must hold that $k^x \not\in \maU_{C_n}^y\left(T_n\right)$ or that $k^y \not\in \maU_{C_n}^x\left(T_n\right)$ (or both), otherwise
      we would be in $\mathscr U_n$. In the first case (the other one is symmetric), we have that $Q^x_{k^x}(n)>0$ and $Q^y_{k^x}(n)=0$, and we are
      in case 2 of the proof of Lemma \ref{lemma:Inclusion1}.
      \item If exactly one of the two sets is empty, say $\maU_{C_n}^y\left(T_n\right)$ is empty and $\maU_{C_n}^x\left(T_n\right)$ is not (the other way around
      is symmetric), then the incoming item at $T_n$ is matched in the system initiated by $x$ and not in the system initiated by $y$, so we are in case 3 of the
      proof of Lemma \ref{lemma:Inclusion1}.
      \item If the two sets are empty, then the incoming item at $T_n$ is matched in none of the two systems, so we are in case 4 of Lemma \ref{lemma:Inclusion1}.\qedhere
      \end{itemize}
\end{enumerate}
\end{proof}

\begin{proof}[Proof of Proposition \ref{propGR}]
Fix a connected graph $G=(\maV,\maE) \in \mathscr{G}_7^c$, and $\Phi=\sU$. Fix a node $i_0 \in \maV$ and denote again
$\maS:= \maV\setminus\left(\{i_0\}\cup \maE(i_0)\right)=\left\{i_1,...,i_{|\maS|}\right\}$.
For any  $j\in \maV$ denote
\[\maS(j)=\left\{\ell \in \llbracket 1,|\maS| \rrbracket\,:\, i_\ell \in \maE(j)\right\},\]
and for any $r \in \llbracket 0,|\maS(j)| \rrbracket$, let
\begin{equation}
\label{eq:defVS} \maV^{\maS}_{j,r} = \Bigl\{x \in \mGS \,:\,
\mbox{Card}\left\{\ell \in \maS(j)\,:\,x_\ell >0\right\} = r \Bigl\},
\end{equation}
where, for notational convenience, we let $\mbox{Card}\,A$ denote
the cardinality of the set $A$.

Suppose that Assumption \ref{AssErgo} holds for the sequence of
processes $\{\chi^n \,:\, n\ge 1\}$ corresponding to the marginal
process $\chi$ of infinitesimal generator
\bequ
\label{XnGeneratorGR} \left\{\begin{array}{ll}
\mathscr{A}^{\maS,\sU}(x, x+\gre_\ell) & = \lm_{i_\ell} \ind_{\maO_{i_\ell}}(x),~ \ell\in \llbracket 1,|\maS| \rrbracket;\\
&\\
\mathscr{A}^{\maS,\sU}(x, x-\gre_\ell) & = \displaystyle\ind_{\maN_\ell^{\maS}}(x)
\sum_{j \in \maE(i_\ell)}\sum_{r=0}^{\maS(j)}{\lm_j \over r} \ind_{\maV^\maS_{j,r}}(x), ~\ell\in \llbracket 1,|\maS| \rrbracket. \\
\end{array}\right.
\eeq
As is easily seen from the definition of the policy $\sU$, and similarly to (\ref{XiEqDistSn}), the process
$\chi^n$ coincides in distribution with the restriction $S^n$ of the process $Q^n$ to its coordinates in $\maS$, as long as $Q^n_{i_0}$ remains strictly positive.
Provided that at time $t$, $Q^n_{i_0}(t)>0$, $S^n(t) \in
\maV^{\maS}_{j,r}$, and an item of a class $j\in \maE(i_0)$ enters
the system, the match of the incoming item is drawn uniformly among
all $r$ classes of $\maS(j)$ having items in line, and the class
$i_0$.
Consequently, under assumption \ref{AssInitial}, an analogous result to Theorem \ref{thFWLLN} holds, with the following drift for the fluid limit of the $i_0$-coordinate,
\bequ \label{eq:driftGR}
\lm_{i_0} - \sum_{j\in \maE\left(i_0\right)}\sum_{r=0}^{\maS(j)} {\lm_j \over r+1}\pi\left(\maV^{\maS}_{j,r}\right).
\eeq

Now, as a consequence of Lemma \ref{lemma:Inclusion1GR},
Lemma \ref{LemInclusion2} still holds true for the matching queues
$\left(G,\lambda,\sU\right)_{\ssc}$ and $(\tilde G,\lambda,\sU)_{\ssc}$, where the disconnected graph $\tilde G$ is constructed from $G$ and any induced
subgraph $\breve G$, as in Figure \ref{Fig:graphtilde}. From Lemma \ref{Lm:BipSepCycles}, $G$ induces a pendant graph or an odd cycle of size $5$,
and we let $\breve G=(\breve{\maV},\breve{\maE})$ be that induced subgraph.
In view of (\ref{eq:driftGR}), provided that we exhibit an arrival-rate vector $\breve{\lambda} \in \left(\R_{++}\right)^{|\breve{\maV}|}$ such that
\begin{equation*}
\beta:=\breve\lm_{i_0} - \sum_{j\in \maE\left(i_0\right)}\sum_{r=0}^{\maS(j)} {\breve\lm_j \over r+1}\pi\left(\maV^{\maS}_{j,r}\right)>0,
\end{equation*}
the matching queue $\left(\breve G,\breve \lambda,\sU\right)_{\ssc}$ is unstable,
and the proof follows the same arguments as the proof of Theorem \ref{thm:main}.
Thus, it remains to prove the existence of an unstable matching queue $\left(\breve G,\breve \lambda,\sU\right)_{\ssc}$
for $\breve G$ the pendant graph or the odd cycle. This is done as follows.

$\underline{\mbox{Pendant Graph.}}$ Set $i_0=4$. For $\Phi=\sU$,
from (\ref{XnGeneratorGR}) the generator of the marginal process
$\chi$ is the same as (\ref{eq:transPendant0}), replacing the
arrival rate $\breve \lm_3 := \lm_3$ to node 3 by $\breve \lm_3/2$.
(We add the `breve' to the notation of the arrival rates since we
are now considering the pendant graph as the induced graph $\breve
G$ in $G$.) Then, similarly to (\ref{eq:defalphaPendant}), we obtain
that
\begin{equation*}
\alpha :=\pi\left((0,0)\right)={(\breve\lambda_3/2)^2-(\breve\lambda_1-\breve\lambda_2)^2 \over (\breve\lm_3/2)(\breve\lambda_3+\breve\lambda_1+\breve\lambda_2)}.
\end{equation*}
The drift in (\ref{eq:driftGR}) reads
\begin{multline}
\breve\lm_{4} - \breve\lm_3 \pi\left((0,0)\right)-{\breve\lm_3 \over 2}\pi\left(\{0\}\times\Z_{++}\right)-{\breve\lm_3 \over 2}\pi\left(\Z_{++}\times \{0\}\right)\\
= \breve\lm_{4} - \breve\lm_3 \alpha - {\breve\lm_3 \over 2}(1-\alpha)= \breve\lm_{4} - {\breve\lm_3 \over 2}(1+\alpha).\label{eq:ethiopunck}
\end{multline}
Fix $\epsilon \in \left(0,7/15\right]$ and set (see Figure \ref{Fig:GRinstable})
\[\left\{\begin{array}{ll}
\breve\lambda_1&=\breve\lambda_2={\epsilon};\\
\breve\lambda_3&={1\over 2}-{\epsilon / 2};\\
\breve\lambda_4&={1\over 2}-{3\epsilon / 4}.
\end{array}\right.\]
It can be easily checked that $\breve \lambda \in \textsc{Ncond}\left(\breve G\right)$.
However, the drift in (\ref{eq:ethiopunck}) becomes
\[{1 \over 2} - {3\varepsilon \over 4}
-{1\over 2}\left({1\over 2} - {\varepsilon \over 2}\right)\left(1+{1/2 - \epsilon/2 \over 1/2 - \epsilon/2 + 4\epsilon}\right)
={\epsilon \over 4\left(1+7\epsilon\right)}\left(7-15\epsilon\right)>0.\]

$\underline{\mbox{5-cycle.}}$ Set $i_0=4$. For $\Phi=\sU$, from (\ref{XnGeneratorGR})
the generator of the marginal process is the same as (\ref{eq:trans5cycle}), replacing
the arrival rates to nodes 3 and 4, $\breve \lm_i = \lm_i$, by $\breve \lm_i/2$, $i=3,4$.
The drift in (\ref{eq:driftGR}) reads
\begin{multline}
\breve\lm_{5} - \breve\lm_3 \tilde\pi\left(\{0\}\times \Z_+\right)-\breve\lm_4 \tilde\pi\left(\Z_+ \times \{0\}\right)\\
- {\breve\lm_3 \over 2}\tilde\pi\left(\Z_{++}\times \{0\}\right) - {\breve\lm_4 \over 2}\tilde\pi\left(\{0\}\times\Z_{++}\right),\label{eq:ethiopunck5cycle}
\end{multline}
where $\tilde\pi$ is the stationary distribution of the fast process, obtained similarly to $\tilde \pi$ in (\ref{pro:5cycle}), replacing
the intensities at nodes 3 and 4, $\breve \lm_i = \lm_i$, by $\breve \lm_i/2$, $i=3,4$.

Now, if we fix $\epsilon \in \left(0,7/23\right]$ and set (see Figure \ref{Fig:GRinstable})
\[\left\{\begin{array}{ll}
\breve\lambda_1 &=\breve\lambda_2={\epsilon};\\
\breve\lambda_3 &=\breve\lambda_4 = {1\over 4}-{\epsilon / 4};\\
\breve\lambda_5 &={1\over 2}-{3\epsilon / 4}.
\end{array}\right.\]
Again, we can easily check that $\breve \lambda \in \textsc{Ncond}\left(\breve G\right)$, and the drift in (\ref{eq:ethiopunck5cycle}) equals
\[{1\over 2} - {3\varepsilon \over 4} - \left({1\over 2} - {\varepsilon \over 2}\right){1+7\varepsilon \over 1+15\varepsilon}
- \left({1\over 4} - {\varepsilon \over 4}\right){8\varepsilon \over 1+15\varepsilon}={\epsilon \over 4\left(1+15\epsilon\right)}\left(7-23\epsilon\right)>0.\] \qedhere
\end{proof}
\begin{figure}[h!]
\begin{center}
\begin{tikzpicture}
\draw[-] (0,2) -- (0,1);
\draw[-] (0,1) -- (-1,0);
\draw[-] (0,1) -- (1,0);
\draw[-] (-1,0) -- (1,0);
\fill (0,2) circle (2pt) node[above] {\small{${1\over 2}-{3\epsilon \over 4}$}} ;
\fill (0,1) circle (2pt) node[above right] {\small{${1\over 2}-{\epsilon \over 2}$}} ;
\fill (-1,0) circle (2pt) node[left] {\small{$\epsilon$}} ;
\fill (1,0) circle (2pt) node[right] {\small{$\epsilon$}} ;
\fill (5,2) circle (2pt) node[above]{{\small{${1\over 2}-{3\epsilon \over 4}$}}};
\fill (4,1) circle (2pt) node[left] {{\small{${1\over 4}-{\epsilon \over 4}$}}};
\fill (6,1) circle (2pt) node[right]{{\small{${1\over 4}-{\epsilon \over 4}$}}};
\fill (4,0) circle (2pt) node[left] {\small{$\epsilon$}};
\fill (6,0) circle (2pt) node[right] {\small{$\epsilon$}};
\draw[-] (4,0) -- (6,0);
\draw[-] (4,0) -- (4,1);
\draw[-] (6,0) -- (6,1);
\draw[-] (4,1) -- (5,2);
\draw[-] (6,1) -- (5,2);
\end{tikzpicture}
\caption[smallcaption]{Unstable uniform matching queues: left, on the pendant graph and right, on the 5-cycle.}\label{Fig:GRinstable}
\end{center}
\end{figure}
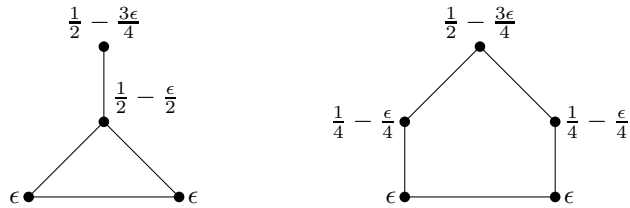

\section{Summary and Future Research} \label{secSummary}

In this paper we proved that matching queues on graphs in $\mathscr{G}_7^c$ satisfying {\sc Ncond} need not be stable.
Our proof employed a fluid-limit whose characterization builds on estimating the stationary distribution of a related marginal process.

There are many directions for future research. We specify four which we are currently investigating.

{\bf Generalizing the result.~} It follows from Lemma
\ref{Lm:BipSepCycles} that any non-bipartite and non-separable graph
induces an odd cycle of size $5$ or more, or the pendant graph. For
both the pendant and the $5$-cycle graphs we have shown that
$\textsc{Ncond}$ is not a sufficient condition for stability
(Propositions \ref{prop:nonsufficientpendant} and
\ref{prop:nonsufficient5cycle}). If a similar instability result
could be shown for any odd-cycle, then the following conjecture
would be proved via an application of the arguments in the proof of
Theorem \ref{thm:main}.

\begin{conjecture} \label{conj}
The {\em only} connected and non-bipartite graphs $G$ for which the matching queue
$(G,\lm,\Phi)_{\ssc}$ is stable for any admissible matching policy $\Phi$ and any $\lambda \in \textsc{Ncond}_{\ssc}(G)$,
are the separable graphs of order $3$ or more.
\end{conjecture}

A direct demonstration of instability of a matching queue on a $p$-cycle (where $p$ is odd)
requires computing the stationary distribution of the associated
$(p-3)$-dimensional marginal process $\chi$. Unfortunately, if $p
\ge 7$, the marginal process is not reversible, so that obtaining
closed-form expressions for the stationary distributions of all
possible $(p-3)$-dimensional CTMC's seems prohibitively hard.
Nevertheless, one might be able to appropriately bound these
stationary distributions and prove a result analogous to
Propositions \ref{prop:nonsufficientpendant} and
\ref{prop:nonsufficient5cycle}.

{\bf Identifying bottlenecks.~}
The fluid limit may be used to construct a procedure determining the ``bottlenecks'' (namely, unstable) nodes of general unstable matching queues.
Moreover, when the stationary distribution of the fast-time-scale CTMC can be computed explicitly,
the fluid limit provides the exact rate of increase of the queues corresponding to the unstable nodes.

{\bf Matching models on hypergraphs. ~} In our model, items depart
the system by pairs. However, in many applications (e.g.,
manufacturing and assemble-to-order systems) matchings can occur in
groups that are larger than $2$ (as, for example, in \cite{GurWa}).
Thus, it remains to establish an analogue to {\sc Ncond} when the
compatibility between items cannot be represented by a graph, but
more generally, by an hypergraph.


\appendix
\section{Matching Algorithms on Random Graphs}
\label{sec:RG}


 In graph theory, a {\em matching on a graph}
$\mathbf G$ is a subgraph $\check {\mathbf G}$ of $\mathbf G$ in
which each node has exactly one neighbor. The matching is said to be
{\em perfect} if $\mathbf G$ and $\check{\mathbf G}$ have the same
set of nodes. It is a consequence of Tutte's Theorem (a
generalization of Hall's Marriage Theorem to arbitrary graphs), that
a necessary condition for the existence of a perfect matching on
$\mathbf G$ is given by
\begin{equation}
\label{eq:tutte} |\mathbf I| \le |\maE_{\mathbf G}\left(\mathbf
I\right)|,\mbox{ for any independent set }\mathbf I \in \mathbb
I(\mathbf G),
\end{equation}
where $\maE_{\mathbf G}\left(\mathbf I\right)$ denotes the set
of neighbors of the elements of $\mathbf I$ in $\mathbf G$.
A {\em matching algorithm} is a procedure for constructing a matching.

It is a well-known fact that, even when a perfect matching exists on
$\mathbf G$, an on-line matching algorithm under which,
at each step a node is chosen uniformly at random among all
unmatched nodes and its match is chosen uniformly at random among
all its unmatched neighbors, fails in general to lead to a perfect
matching.

A matching on a graph is not to be confused with the stochastic
matching of items discussed thus far. However, clear connections can
be drawn between the two problems, as we briefly illustrate by a
simple example.

 Consider a {\em random} graph in which the nodes are of $p$
different types, such that the types of the various nodes are random
and i.i.d., having a common distribution $\mu$ on the set of types
$\llbracket 1,p \rrbracket$. Nodes of the same type are not
neighbors of each other, and have the same neighbors. Specifically,
we fix a given auxiliary (simple) graph $G=(\maV,\maE)$ of size $p$,
which we call {\em template graph}, and the set of types of
the nodes of $\mathbf G$ is identified with $\maV$. The edges of
$\mathbf G$ are fully determined by the types of the nodes according
to the following rule: two nodes $u$ and $v$ of $\mathbf G$, of
respective types $i$ and $j$, are neighbors in $\mathbf G$ if and
only if $i\v j$ in $G$.

We aim to construct a matching on the resulting graph $\mathbf G$.  
Our approach is to construct the random graph $\mathbf G$ {\em
together} with the matching on $\mathbf G$ sequentially. To this
end, we define $p$ independent Poisson processes $N_1,...,N_p$ with
respective intensities $\mu(i),\,i\in\llbracket 1,p \rrbracket$, and
let $T_1,T_2,...$ the points of the superposition $N$ of the $p$
processes. We also fix a matching queue $(G, \mu, \Phi)_{\ssc}$,
where the arrival-rate vector is denoted
$\mu=\left(\mu(1),...,\mu(p)\right)$. We proceed by induction, at
each point of $N$. For all $n\ge 1$,
\begin{itemize}
\item[(i)] Let $i$ be the element of $\llbracket 1,p \rrbracket$ such that $T_n$ is a point of $N_i$. At $T_n$, create a node $u$ of $\mathbf G$, and assign to $u$ the type $i$.
Then, create an edge in $\mathbf G$ between $u$ and all the
previously created nodes of all types $j$ such that $i \v j$ in the
template graph $G$.
\item[(ii)] If the set of unmatched neighbors of $u$ in $\mathbf G$ is non-empty, apply $\Phi$ to select a unique node $v$ in the latter
set, exactly as we choose a match for an item of class $i$ in the
matching queue $(G, \mu, \Phi)_{\ssc}$. We call $v$ the {\bf match}
of $u$, and say that both $u$ and $v$ are {\em matched} nodes. If no
neighbor of $u$ is unmatched, we set the status of $u$ to be {\em
unmatched}.
\end{itemize}
At any time $T_n$, the graph $\mathbf G$ has exactly $n$ nodes, some
of which are matched and the other are unmatched. We let
$\mathscr M(\mathbf G)$ and by $\bar{\mathscr M}(\mathbf G)$ denote the sets of matched and unmatched nodes of $\mathbf G$, respectively.
The {\em matching} on $\mathbf G$ is the set of nodes $\mathscr M(\mathbf G)$, together with the edges between matched couples.

Almost surely for a large enough $n$, the resulting graph $\mathbf
G$ at $T_n$ is $p$-partite (all types of nodes are represented in
$\mathbf G$, and there is no edge between any two nodes of the same
type). Moreover, as a straightforward application of the Strong Law of Large Numbers (SLLN),
the proportions of nodes of the various types tend
to the measure $\mu$ as $n$ increases to infinity.

The construction just described is
related to the procedure of uniform random pairing; see
\cite{Wormald}. This latter procedure leads to the so-called
configuration model introduced by Bollobas \cite{Bollo} (see also
\cite{Remco}), generating a realization of a random graph where the
{\em degree}, i.e., the number of neighbors of each node, is fixed
beforehand. We also refer to \cite{GG09} and the
references therein, for general results concerning matching on
random graphs.

\begin{figure}[h!]
\begin{center}
\begin{tikzpicture}
\fill (0,1) circle (2pt) node[left]{$1$}; \fill (2,1) circle (2pt)
node[right]{$2$}; \fill (1,2) circle (2pt) node[right]{$3$}; \fill
(1,3) circle (2pt) node[above]{$4$}; \draw[-] (0,1) -- (2,1);
\draw[-] (0,1) -- (1,2); \draw[-] (1,2) -- (2,1); \draw[-] (1,2) --
(1,3);
\fill (5,2) circle (2pt) node[below left]{$2$}; \fill (5,0) circle
(2pt) node[below]{$2$}; \fill (3,1.5) circle (2pt) node[left]{$4$};
\fill (7,3) circle (2pt) node[right]{$3$}; \fill (6,4)circle (2pt)
node[above]{$1$}; \fill (5.5,1) circle (2pt) node[above]{$3$}; \fill
(6.5,1.5) circle (2pt) node[right]{$1$}; \fill (4,3) circle (2pt)
node[above]{$3$}; \draw[-, very thin] (4,3) -- (5,2); \draw[-, very
thin] (7,3) -- (3,1.5); \draw[-,very thick] (7,3) -- (5,0); \draw[-,
very thin] (6,4) -- (5,0); \draw[-, very thin] (5.5,1) -- (3,1.5);
\draw[-, very thick] (5.5,1) -- (6.5,1.5); \draw[-, very thin]
(5.5,1) -- (5,0); \draw[-, very thin] (6.5,1.5) -- (5.5,1); \draw[-,
very thin] (6.5,1.5) -- (5,0); \draw[-, very thin] (6.5,1.5) --
(7,3); \draw[-, very thin] (6,4) -- (7,3); \draw[-, very thin] (4,3)
-- (5,0); \draw[-, very thick] (4,3) -- (3,1.5); \draw[-, very thin]
(4,3) -- (6,4); \draw[-, very thin] (4,3) -- (6.5,1.5);
\draw[-, very thin] (5,2) -- (5.5,1); \draw[-, very thick] (5,2) --
(6,4); \draw[-, very thin] (5,2) -- (6.5,1.5); \draw[-, very thin]
(5,2) -- (7,3);

\end{tikzpicture}
\caption[smallcaption]{Construction of a matching on a 4-partite
graph having the pendant graph as template.}
\label{fig:exampleMatching}
\end{center}
\end{figure}
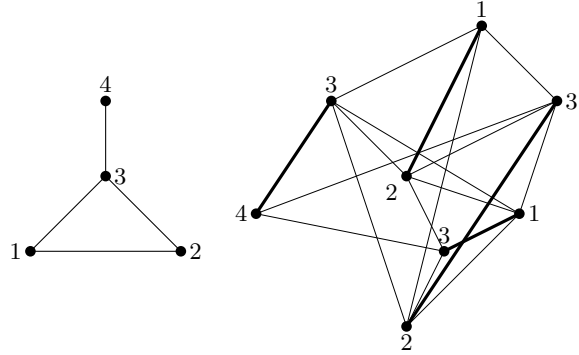

It is then easy to couple the construction described above with the
matching queue $(G,\mu,\Phi)$: if both are constructed with the same
Poisson processes, then

\noindent (i) the creation of a node of type $i$ in the random graph $\mathbf G$ corresponds
to the arrival of a class-$i$ item in the matching queue;

\noindent (ii) a matching between two nodes of respective types $i$ and $j$
in $\mathbf G$ occurs if and only if in the matching queue, at the
same instant, two items of respective classes $i$ and $j$ are
matched, and depart the system.

Consequently, at any time the list of classes of queued items in the
matching queue coincides with the list of types of the unmatched
nodes in $\mathbf G$.

For all $j \in \llbracket 1,p \rrbracket$ and $t \ge 0$,
let $\mathbf Q_j(t)$ denote the number of unmatched nodes of type $j$ in $\mathbf G$ at time $t$, and let
$\mathbf Q=\left(\mathbf Q_1,...,\mathbf Q_p\right).$
As before, let $Q$ denote the queue process of the matching queue in \eqref{eq:defQ}.
It follows that, if $Q(0) = \mathbf Q (0)$, then
\[\mathbf Q_j(t) = Q_j(t),\,j \in \llbracket 1,p \rrbracket,\,t\ge 0.\]
Therefore, Theorems \ref{thm:stabCont} and \ref{propGR} imply the
following.

\begin{corollary}
\label{pro:stabRG} Assume that the template graph $G$ is connected
and non-bipartite.
\begin{itemize}
\item[(i)] If $\Phi=\textsc{ml}$, then for all $\mu \in \mbox{{\sc Ncond}}_{\ssc}(G)$, the Markov process $\mathbf Q$ is positive
recurrent;
\item[(ii)] If $\Phi=\sU$ and $G \in \mathscr G_7^c$, then the Markov process $\mathbf Q$ is positive recurrent for all
$\mu \in \mbox{{\sc Ncond}}_{\ssc}(G)$ if and only if $G$ is separable. 
\end{itemize}
\end{corollary}
Let us now observe that just after time $T_n$, the size of the
matching on $\mathbf G$ is given by
\begin{equation}
\label{eq:sizematching} |\mathscr M(\mathbf G)|=n - |\bar{\mathscr
M}(\mathbf G)|=n-\sum_{i=1}^p \mathbf Q_i\left(T_n\right).
\end{equation}

In view of (\ref{eq:sizematching}), 
Corollary \ref{pro:stabRG} suggests that the
separable graphs are the only template graphs in $\mathscr G_7^c$
guaranteeing that under {\sc Ncond} and the uniform
policy, $|\bar{\mathscr M}(\mathbf G)|$ becomes negligible with
respect to $n$ as $n$ increases. Therefore, the size of the matching and the size of the
graph tend to coincide.

Now, for any time $T_n$ and for any set of nodes $A \subset \maV$,
let $\mathbf X_A(T_n)$ denote the set of nodes of $\mathbf G$ at $T_n$ having types in $A$. We have
\begin{equation}
\label{eq:loser1} \mathbf X_A(T_n)= \sum_{j \in A}
N_j\left(\left(0,T_n\right]\right),
\end{equation}
where $N_j((0,T_n])$ is the number of points of $N_j$ up to time
$T_n$. By construction, for any two nodes $u$ and $v$ in
$\mathbf G$ of respective types $i$ and $j$, $u \pv v$ in $\mathbf
G$ entails $i \pv j$ in $G$. Thus, for any independent set $\mathbf
I$ of $\mathbf G$ at $T_n$, there exists a unique independent set
$\maI$ of $G$, such that $\mathbf I \subset \mathbf X_{\maI}(T_n)$,
and all types in $\maI$ are represented in $\mathbf I$.  Moreover,
the set of all the neighbors in $\mathbf G$ of the elements
of $\mathbf I$ is exactly the set of all nodes of $\mathbf G$ having
types that are neighbors in $G$ of the types belonging to
$\maI$, that is,
\begin{equation}
\label{eq:loser2}
\mathcal E_{\mathbf G}(\mathbf I)=\mathbf
X_{\maE(\maI)}\left(T_n\right).
\end{equation}
It follows from  (\ref{eq:tutte}), (\ref{eq:loser1}) and (\ref{eq:loser2}) that
\begin{equation}
\label{eq:loser3} \sum_{i \in \mathcal I} N_i\left(T_n\right) \le
\sum_{j \in \mathcal E(\mathcal I)}N_j\left(T_n\right),\,\mbox{ for
all independents sets }\maI \in \mathbb I(G).
\end{equation}
is a necessary condition for the existence of a perfect matching on $\mathbf G$ at time $T_n$.
Dividing both sides of the equality in (\ref{eq:loser3}) by $n$ and taking $n$ to
infinity, the SLLN implies that
$\mu\left(\mathcal I\right) \le \mu\left(\mathcal E(\mathcal I)\right)$, for all independents sets $\maI \in \mathbb I(G).$

We conclude that a necessary condition for the existence of a
perfect matching on $\mathbf G$ in the large graph limit is closely
related to {\sc Ncond}$(G)$. (Specifically, the strict inequality in {\sc Ncond} is replaced by a weak inequality).
Thus, Corollary \ref{pro:stabRG} is reminiscent of the aforementioned result
concerning the construction of matchings using uniform on-line
algorithms: aside from the case of separable graphs (for which all
matching policies are equivalent in terms of types, see Section
\ref{subsecPre}), under a condition on the connectivity of $\mathbf
G$ that is closely related to (\ref{eq:tutte}), a matching policy
that is uniform in terms of types of nodes also may fail in general
to construct a perfect matching on $\mathbf G$.

\providecommand{\bysame}{\leavevmode\hbox to3em{\hrulefill}\thinspace}
\providecommand{\MR}{\relax\ifhmode\unskip\space\fi MR }
\providecommand{\MRhref}[2]{%
  \href{http://www.ams.org/mathscinet-getitem?mr=#1}{#2}
}
\providecommand{\href}[2]{#2}

\end{document}